\documentclass[oneside,11pt]{amsart}
\usepackage{amsmath}
\usepackage{amssymb}
\usepackage{graphicx}
\usepackage{tikz}

\newtheorem{thm}{Theorem}[section]
\newtheorem{prop}[thm]{Proposition}
\newtheorem{lem}[thm]{Lemma}
\newtheorem{cor}[thm]{Corollary}

\theoremstyle{definition}
\newtheorem{defn}[thm]{Definition}
\newtheorem{rem}[thm]{Remark}

\newtheorem{ques}[thm]{Question}

\newcommand{\abs}[1]{\lvert{#1}\rvert}
\newcommand{\bigabs}[1]{\bigl|{#1}\bigr|}

\renewcommand{\bar}[1]{\overline{#1}}
\newcommand{\boundary}{\partial}

\newcommand{\bigset}[2]{ \bigl\{ \, {#1} \bigm| {#2} \, \bigr\} }
\renewcommand{\emptyset}{\varnothing}

\newcommand{\field}[1]{\mathbb{#1}}
\newcommand{\Z}{\field{Z}}
\newcommand{\R}{\field{R}}
\newcommand{\Q}{\field{Q}}
\newcommand{\N}{\field{N}}

\DeclareMathOperator{\Aut}{Aut}

\DeclareMathOperator{\dev}{dev}
\DeclareMathOperator{\CAT}{CAT}

\DeclareMathOperator{\diam}{diam}
\usepackage{ifthen}

\newcommand{\showcomments}{yes}

\newsavebox{\commentbox}
{\ifthenelse{\equal{\showcomments}{yes}}%
{\footnotemark
    \begin{lrbox}{\commentbox}
    \begin{minipage}[t]{1.25in}\raggedright\sffamily\upshape\tiny
    \footnotemark[\arabic{footnote}]}%
{\begin{lrbox}{\commentbox}}}%
{\ifthenelse{\equal{\showcomments}{yes}}%
{\end{minipage}\end{lrbox}\marginpar{\usebox{\commentbox}}}%
{\end{lrbox}}}

\begin{document}

\title{Distortion of surfaces in graph manifolds}

\author{G. Christopher Hruska and Hoang Thanh Nguyen}
\address{Department of Mathematical Sciences\\
University of Wisconsin--Milwaukee\\
P.O.~Box 413\\
Milwaukee, WI 53201\\
USA}
\email{chruska@uwm.edu, nguyen36@uwm.edu}

\date{\today}

\begin{abstract}
Let $S \looparrowright N$ be an immersed  horizontal surface in a $3$--dimensional graph manifold.
We show that the fundamental group of the surface $S$ is quadratically distorted whenever the surface is virtually embedded (\emph{i.e.}, separable) and is exponentially distorted when the surface is not virtually embedded.
\end{abstract}

 % Hyperbolic groups and nonpositively curved groups


\maketitle

\section{Introduction}
\label{sec:Introduction}

In the study of $3$--manifolds, much attention has focused on the distinction between surfaces that lift to an embedding in a finite cover and those that do not.
A $\pi_1$--injective immersion $S \looparrowright N$ of a surface $S$ into a $3$--manifold $N$ is a \emph{virtual embedding} if (after applying a homotopy) the immersion lifts to an embedding of $S$ into a finite cover of $N$.
Due to work of Scott and Przytycki--Wise \cite{Scott78,PrzytyckiWise14Separability}, virtual embedding is equivalent to separability of the surface subgroup $\pi_1(S)$ in $\pi_1(N)$.

A major part of the solution of the virtual Haken conjecture is Wise and Agol's theorem that every immersed surface in a finite volume complete hyperbolic $3$--manifold is virtually embedded \cite{WiseHierarchies,Agol13}.
In contrast, Rubinstein--Wang \cite{RW98} constructed non--virtually embedded surfaces $g\colon S \looparrowright N$ in many $3$--dimensional graph manifolds $N$.  The examples of Rubinstein--Wang are \emph{horizontal} in the sense that in each Seifert component $M$ of $N$, the intersection $g(S)\cap{M}$ is transverse to the Seifert fibration. Furthermore, they introduced a combinatorial invariant of horizontal surfaces called dilation, which they use to completely characterize which horizontal surfaces are virtually embedded.

If $H \le G$ are groups with finite generating sets $\mathcal{S}$ and $\mathcal{T}$,
the \emph{distortion} of $H$ in $G$ is given by
\[
   \Delta_{H}^G(n)
     = \max \bigset{\abs{h}_{\mathcal{T}}}{\text{$h\in H$ and $\abs{h}_{\mathcal{S}}\leq n$} }.
\]
Distortion does not depend on the choice of finite generating sets $\mathcal{S}$ and $\mathcal{T}$ (up to a natural equivalence relation). 
The purpose of this article is to address the following question of Dani Wise:

\begin{ques}
\label{ques:Wise}
Given a $3$--dimensional graph manifold $N$, which surfaces in $N$ have nontrivial distortion?
\end{ques}

Question~\ref{ques:Wise} arises naturally in the study of cubulations of $3$--manifold groups. The typical strategy for constructing an action of the fundamental group on a $\CAT(0)$ cube complex is to find a suitable collection of immersed surfaces and then to consider the $\CAT(0)$ cube complex dual to that collection of surfaces (see for instance \cite{Wise12}).

Whenever a group $G$ acts properly and cocompactly on a $\CAT(0)$ cube complex $X$, the stabilizer of each hyperplane must be an undistorted subgroup of $G$, since hyperplanes are convex.  
Hagen--Przytycki \cite{HagenPrzytycki15} show that chargeless graph manifolds do act cocompactly on $\CAT(0)$ cube complexes.
It is clear from their construction that many graph manifolds contain undistorted surface subgroups.

However the situation for horizontal surfaces turns out to be quite different.
Our main result states that horizontal surfaces in graph manifolds always have a nontrivial distortion, and this distortion is directly related to virtual embedding in the following sense:

\begin{thm}
\label{thm:main thm}
Let $S \looparrowright N$ be a horizontal surface properly immersed in a graph manifold $N$. The distortion of $\pi_1(S)$ in $\pi_1(N)$ is quadratic if $S$ is virtually embedded, and exponential if $S$ is not virtually embedded.
\end{thm}

The main tool used in the proof of this theorem is a simple geometric interpretation of Rubinstein--Wang's dilation in terms of slopes of lines in the JSJ planes of the universal cover.

We note that, throughout this paper, the term ``graph manifold'' specifically excludes two trivial cases: Seifert manifolds and Sol manifolds.
By Proposition~\ref{prop:linear distortion} horizontal surfaces in Seifert manifolds are always undistorted.  Distortion of surfaces in Sol manifolds is not addressed in this article.

Although Theorem~\ref{thm:main thm} deals only with horizontal surfaces in graph manifolds, understanding these is the key to a general understanding of all $\pi_1$--injective surfaces in nongeometric $3$--manifolds.
The nongeometric $3$--manifolds include both graph manifolds and also ``mixed'' type $3$--manifolds, those whose JSJ decomposition contains at least one hyperbolic component and at least one JSJ torus.

Recently Yi Liu has used Rubinstein--Wang's work on virtual embedding of horizontal surfaces in graph manifolds as the foundation for a study of virtual embedding of arbitrary surfaces in nongeometric $3$--manifolds \cite{Liu17}.
Similarly, in a forthcoming article, the second author uses Theorem~\ref{thm:main thm} as the foundation for a study of distortion of arbitrary surface subgroups in fundamental groups of nongeometric $3$--manifolds \cite{NguyenDistortion}.

As mentioned above,
Hagen--Przytycki \cite{HagenPrzytycki15} 
%and Tidmore \cite{Tidmore}
have shown that chargeless graph manifolds
%and chargeless mixed manifolds
act cocompactly on $\CAT(0)$ cube complexes. The cubulation they construct is dual to a family of properly immersed surfaces, none of which is entirely horizontal.  
More precisely a key property of these surfaces is that they never contain two adjacent horizontal pieces.

The following corollary shows that in order to obtain a proper, cocompact cubulation (using any possible subgroups, not necessarily just surface subgroups) all surface subgroups must be of the type used by Hagen--Przytycki.  The corollary follows from combining Theorem~\ref{thm:main thm} with \cite{NguyenDistortion}.

\begin{cor}
Let $G$ be the fundamental group of a graph manifold.  Let $\{H_1,\dots,H_k\}$ be a collection of codimension--$1$ subgroups of $G$.
Let $X$ be the corresponding dual $\CAT(0)$ cube complex.
If at least one $H_i$ is the fundamental group of a surface containing two adjacent horizontal pieces, then the action of $G$ on $X$ is not proper and cocompact.
\end{cor}

%An analogous result could also be obtained in the setting of mixed manifolds using Corollary~\ref{cor:MixedDistortion}.

After seeing an early version of this paper, Hung Cong Tran discovered an alternate proof of the quadratic distortion of certain horizontal surfaces whose fundamental groups can also be regarded as Bestvina--Brady kernels.  This observation is part of a broader study by Tran of the distortion of Bestvina--Brady kernels in right-angled Artin groups \cite{Tran16}.

\subsection{Connections to previous work}

In \cite{Woodhouse16}, Woodhouse exploited strong parallels between graph manifolds and tubular spaces to study actions of tubular groups on $\CAT(0)$ cube complexes.
In the tubular setting, immersed hyperplanes play the role of immersed surfaces in graph manifolds
Woodhouse extended the Rubinstein--Wang theory of dilation and found a connection between dilation and distortion. In particular, he proves that if an immersed hyperplane in a tubular group has nontrivial dilation then its distortion is at least quadratic. Inspired by Woodhouse's work, we use analogous techniques in the cleaner geometric setting of graph manifolds to obtain Theorem~\ref{thm:main thm}.

We remark that the main proof of Theorem~\ref{thm:main thm} in Sections \ref{sec:LowerBound}~and~\ref{sec:UpperBound} includes both virtually embedded and non--virtually embedded surfaces in a unified treatment that prominently uses a simple geometric interpretation of dilation in terms of ``slopes of lines'' in the Euclidean geometry of JSJ planes.
This interpretation was inspired by Woodhouse's earlier work on tubular groups.  It seems likely that the main proofs here could be translated back to the tubular setting, where they may lead to new advances in that setting as well.

However there is an alternate (shorter) proof of the quadratic distortion of virtually embedded horizontal surfaces using the following result, which combines work of Gersten and Kapovich--Leeb.  This alternate approach, while more direct, does not give any information about the structure of non--virtually embedded surfaces.

\begin{thm}[Gersten, Kapovich--Leeb]
\label{thm:FiberQuadratic}
Let $N$ be a graph manifold that fibers over the circle with fiber surface $S$.
Then $\pi_1(S)$ is quadratically distorted in $\pi_1(N)$.
\end{thm}

Kapovich--Leeb implicitly use the quadratic upper bound (without stating it explicitly) in \cite{KapovichLeeb98}, where they attribute it to Gersten \cite{Gersten94Quadratic}.  The quadratic lower bound also follows easily from results of Gersten and Kapovich--Leeb \cite{Gersten94Quadratic,KapovichLeeb98}, but was not specifically mentioned by them.
The discussion of Theorem~\ref{thm:FiberQuadratic} in \cite{KapovichLeeb98} and its precise derivation from \cite{Gersten94Quadratic} is brief and was not the main purpose of either article.
For the benefit of the reader we have included a more detailed exposition of Kapovich--Leeb's elegant proof of Theorem~\ref{thm:FiberQuadratic} in Section~\ref{sec:FiberQuadratic}, which relies on Thurston's geometric classification of $3$--manifolds that fiber over the circle.

The virtually embedded case of Theorem~\ref{thm:main thm} can be derived as a corollary of Theorem~\ref{thm:FiberQuadratic} as follows.
If a horizontal surface $g \colon S \looparrowright N$ is virtually embedded then there exist finite covers $\hat{S}\to S$ and $\hat{N}\to N$ such that $\hat{N}$ is an $\hat{S}$--bundle over $S^1$ (see for instance \cite{WangYu}).  By Theorem~\ref{thm:FiberQuadratic}, the distortion of $\pi_1(\hat{S})$ in $\pi_1(\hat{N})$ is quadratic.  Distortion is unchanged when passing to subgroups of finite index, so the distortion of $\pi_1(S)$ in $\pi_1(N)$ is also quadratic.

\subsection{Overview}
In Section~\ref{prelininaries} we review some concepts in geometric group theory.
In Section~\ref{sec:surfaces} we give several lemmas about curves on hyperbolic surfaces that will be used in Section~\ref{sec:LowerBound}.
Section~\ref{sec:GraphManifolds} is a review background about graph manifolds and horizontal surfaces. A convenient metric on a graph manifold that will use in this paper will be discussed. 

In Section~\ref{sec:LowerBound} we prove the distortion of a horizontal surface in a graph manifold is at least quadratic. We also show that if the horizontal surface is not virtually embedded then the distortion is at least exponential. The strategy in this proof was inspired by the work of Woodhouse (see Section~6 in \cite{Woodhouse16}).
In Section~\ref{sec:UpperBound} we prove that distortion of a horizontal surface in a graph manifold is at most exponential. Furthermore, if the horizontal surface is virtually embedded then the distortion is at most quadratic.

Section~\ref{sec:FiberQuadratic} contains a detailed exposition of the proof of Theorem~\ref{thm:FiberQuadratic}.

\subsection{Acknowledgments}
The authors are grateful to Hung Cong Tran for suggesting the possibility of a connection between the virtually embedded case of Theorem~\ref{thm:main thm} and the earlier work of Gersten on quadratic divergence.
The first author wishes to thank Noel Brady for long conversations about distortion of subgroups and divergence of $3$--manifold groups many years ago.  These conversations when the first author was a graduate student had a long delayed---but substantial---influence over the content of this article.

The authors also benefited from many helpful conversations about this material with Boris Okun, Joseph Tidmore, Becca Winarski, Dani Wise, and Daniel Woodhouse.  The authors are grateful for the insightful and detailed critiques of multiple referees that have helped improve the exposition of this paper.

This work was partially supported by a grant from the Simons Foundation (\#318815 to G. Christopher Hruska).

%%%%%%%%%%%%%%%%%%%%%%%%%%%%%%%%%%%%%%%%%
\section{Preliminaries}
\label{prelininaries}
%%%%%%%%%%%%%%%%%%%%%%%%%%%%%%%%%%%%%%%%%%%%%%%%

In this section, we review some concepts in geometric group theory: quasi--isometry, distortion of a subgroup and the notions of domination and equivalence.

\begin{defn}
Let $(X,d)$ be a metric space. A path $\gamma \colon [a,b] \to X$ is a \emph{geodesic} if $d(\gamma(s),\gamma(t))=\abs{s-t}$ for all $s,t \in [a,b]$.
A simple loop $f\colon S^1 \to X$ is a \emph{geodesic loop} if $f$ is an isometric embedding with respect to some length metric on $S^1$.
\end{defn}

\begin{defn}[quasi-isometry] 
Let $(X_{1},d_{1})$ and $(X_{2},d_{2})$ be metric spaces. A (not necessarily continuous) map $f \colon X_{1} \to X_{2}$ is an \emph{$(L,C)$--quasi-isometric embedding} if there exist constants $L\geq 1$ and $C\geq 0$ such that for all $x,y \in X_{1}$ we have
\[
   \frac{1}{L} \, d_{1}(x,y) - C
   \leq d_{2} \bigl( f(x),f(y) \bigr)
   \leq L\,d_1(x,y) + C.
\]
If, in addition, there exits a constant $D\geq{0}$ such that every point of $X_{2}$ lies in the $D$--neighborhood of the image of $f$, then $f$ is an \emph{$(L,C)$--quasi-isometry}. When such a map exists, $X_{1}$ and $X_{2}$ are \emph{quasi-isometric}.
\end{defn}

Let $(X,d)$ be a metric space, and $\gamma$ a path in $X$. We denote the length of $\gamma$ by $\abs{\gamma}$.

\begin{defn}[quasigeodesic]
A path $\gamma$ in a metric space $(X,d)$ is an \emph{$(L,C)$--quasigeodesic} with respect to constants $L\geq 1$ and $C\geq 0$ if $\abs{\gamma_{[x,y]} } \leq {L\,d(x,y)+C}$ for all $x,y\in{\gamma}$. A \emph{quasigeodesic} is a path that is $(L,C)$--quasigeodesic for some $L$ and $C$.
\end{defn}

\begin{defn}
A geodesic space $(X,d)$ is \emph{$\delta$--hyperbolic} if every geodesic triangle with vertices in $X$ is \emph{$\delta$--thin} in the sense that each side lies in the $\delta$--neighborhood of the union of the other two sides.
\end{defn}

\begin{defn}[deviation]
Let $(X,d)$ be a geodesic space and $c\geq{0}$ be a fixed number.  Consider a pair of geodesic segments $[x,y]$ and $[z,t]$ such that $d(y,z)\leq{c}$. The \emph{$c$--deviation} of $[x,y]$ and $[z,t]$, denoted $\dev_{c} \bigl([x,y],[z,t] \bigr)$, is the quantity
\[
   \sup \bigset{\max\{d(u,y),d(v,z)\}}{\text{$u\in[x,y]$,$v\in[z,t]$,$d(u,v)\leq{c}$}}.
\]
\end{defn}

The following theorem gives a criterion for determining that a piecewise geodesic in a $\delta$--hyperbolic space is a quasigeodesic.  The proof, which is very similar to the proof of Lemma~19 of \cite[Chapter 5]{GH90}, is left as an exercise to the reader.

\begin{thm}
\label{thm:Quasigeodesic}
Let $(X,d)$ be a $\delta$--hyperbolic space. For any $\kappa\geq{0}$ and $D \ge 0$, there exist constants $L=L(\delta,\kappa,D)$, and $C=C(\delta,\kappa,D)$ such that the following holds: Suppose a piecewise geodesic
\[
   c=[x_1,x_2]\cup[x_2,x_3]\cup\cdots\cup[x_{2m-1},x_{2m}]
\]
satisfies
\begin{enumerate}
   \item
   \label{item:quasigeodesic:Short}
   $d(x_{2i},x_{2i+1})\leq{\kappa}$
   \item
   \label{item:quasigeodesic:Long}
   $d(x_{2i-1},x_{2i})\geq{11\kappa+25\delta+2D}$, and
   \item
   \label{item:quasigeodesic:deviation}
   $\dev_{Q} \bigl( [x_{2i-1},x_{2i}], [x_{2i+1},x_{2i+2}] \bigr) \leq{D}$, where $Q = 4\delta + \kappa$.
\end{enumerate}
Then $c$ is an $(L,C)$--quasigeodesic. \qed
\end{thm}

\begin{defn}
\label{defn:quivalence relation}
Let $f,g$ be functions from positive reals to positive reals. The function $f$ is \emph{dominated} by $g$, denoted $f\preceq g$, if there are positive constants $A$, $B$, $C$, $D$ and $E$ such that
\[
  f(x) \leq A\,g(Bx+C)+Dx+E \quad \text{for all $x$.}
\]
The functions $f$ and $g$ are \emph{equivalent},
denoted $f\sim g$, if $f\preceq g$ and $g\preceq f$.
\end{defn}

The relation $\preceq$ is an equivalence relation. Polynomial functions with degree at least one are equivalent if and only if they have the same degree.  Furthermore, all exponential functions are equivalent.

\begin{defn}[Subgroup distortion]
Let $H\leq{G}$ be a pair of groups with finite generating sets $\mathcal{T}$ and $\mathcal{S}$ respectively. The \emph{distortion} of $(H,\mathcal{T})$ in $(G,\mathcal{S})$ is the function
\[
   \Delta_{H}^G(n)
     = \max \bigset{\abs{h}_{\mathcal{T}}}{\text{$h\in{H}$ and $\abs{h}_{\mathcal{S}}\leq{n}$} }
\]
Up to equivalence, the function $\Delta_{H}^{G}$ does not depend on the choice of finite generating sets $\mathcal{S}$ and $\mathcal{A}$.
\end{defn}

Let $H\leq{G}$ be a pair of finitely generated groups. We say that the distortion $\Delta_{H}^{G}$ is \emph{at least quadratic} [\emph{exponential}] if it dominates a quadratic polynomial [exponential function]. The distortion $\Delta_{H}^{G}$ is \emph{at most quadratic} [\emph{exponential}] if it is dominated by a quadratic polynomial [exponential function].

The following proposition is routine, and we leave the proof as an exercise for the reader.

\begin{prop}
\label{prop:distortion}
Let $G$, $H$, $K$ be finitely generated groups with $K \le H \le G$.
\begin{enumerate}
   \item If $H$ is a finite index subgroup of $G$ then $\Delta_{K}^{H} \sim \Delta_{K}^{G}$.
   \item If $K$ is a finite index subgroup of $H$ then $\Delta_{K}^{G} \sim \Delta_{H}^{G}$.   \qed
\end{enumerate}
\end{prop}

It is well known that a group acting properly, cocompactly, and isometrically on a geodesic space is quasi-isometric to the space.
The following corollary of this fact allows us to compute distortion using the geometries of spaces in place of word metrics.

\begin{cor}
\label{cor:GeometricDistortion}
Let $X$ and $Y$ be compact geodesic spaces, and let $g\colon{(Y,y_0)} \to (X,x_0)$ be $\pi_1$--injective. We lift the metrics on $X$ and $Y$ to geodesic metrics on the universal covers $\tilde{X}$ and $\tilde{Y}$ respectively. Let $G=\pi_1(X,x_0)$ and $H=g_{*} \bigl( \pi_1(Y,y_{0}) \bigr)$.
Then the distortion $\Delta^G_H$ is equivalent to the function
\[
   f(n) = \max \bigset{d_{\tilde Y}(\tilde{y}_0,h (\tilde{y}_0))}{\text{$h \in H$ and $d_{\tilde X}(\tilde{x}_0,h (\tilde{x}_0)) \le n$}}.
\]
\end{cor}

%%%%%%%%%%%%%%%%%%%%%%%%%%%%%%%%%%%%%%%%%%%%%%%%%%%
\section{Curves in hyperbolic surfaces}
\label{sec:surfaces}
%%%%%%%%%%%%%%%%%%%%%%%%%%%%%%%%%%%%%%%%%%%%%%%%%%%

In this section, we will give some results about curves on surfaces that will play an essential role in the proof of Theorem~\ref{thm:lower}.

\begin{defn}[Multicurves]
A closed curve in a surface $S$ is \emph{essential} if it is not freely homotopic to a point or a boundary component.
A \emph{multicurve} in $S$ is a finite collection of disjoint, essential simple closed curves such that no two are freely homotopic.  If $S$ has a metric, a multicurve is \emph{geodesic} if each member of the family is a geodesic loop.
\end{defn}

\begin{lem}
\label{lem:DeviationTree}
Let $S$ be a compact surface with negative Euler characteristic. Let $\mathcal{C}$ be a multicurve in $S$. Let $\mathcal{L}$ be the family of lines that are lifts of loops of $\mathcal{C}$ or boundary loops of $S$.
Equip $S$ with any length metric $d_{S}$. Let $d$ be the induced metric on the universal cover $\tilde{S}$. For any $r>0$ there exists $D=D(r) < \infty$ such that for any two disjoint lines $\ell_1$ and $\ell_2$ of $\mathcal{L}$ we have
\[
  \diam \bigl( \mathcal{N}_{r}(\ell_1) \cap \mathcal{N}_{r}(\ell_2) \bigr) \leq D
\]
\end{lem}

\begin{proof}
Since $H = \pi_1(S)$ acts cocompactly on the universal cover $\tilde{S}$, there exists a closed ball $\overline{B}(x_0,R)$ whose $H$--translates cover $\tilde{S}$.
Note that $\mathcal{L}$ is locally finite in the sense that only finitely many lines of $\mathcal{L}$ intersect the closed ball $B = \overline{B}(x_0,R+2r)$.
Since distinct lines of $\mathcal{L}$ are not parallel, there exists a finite upper bound $D = D(r)$ on the diameter of the intersection
$\mathcal{N}_{r}(\ell)\cap{\mathcal{N}_{r}(\ell')}$ for all lines $\ell\ne \ell' \in \mathcal{L}$ that intersect $B$.

Consider $\ell_1\ne \ell_2 \in \mathcal{L}$.
If $d(\ell_1,\ell_2)\geq{2r}$ then their $r$--neighborhoods have empty intersection, and the result is vacuously true. Thus it suffices to assume that $d(\ell_1,\ell_2)<2r$.
By cocompactness, there exists $h \in H$ so that $h(\ell_1)$ intersects $\overline{B}(x_0,R)$.  But then $h(\ell_1)$ and $h(\ell_2)$ both intersect $B$, so that
\[
   \diam \bigl( \mathcal{N}_r (\ell_1)
    \cap \mathcal{N}_r (\ell_2) \bigr) \leq D,
\]
as desired.
\end{proof}

\begin{lem}
\label{lem:ExistenceOfLoop}
Let $S$ be a compact hyperbolic surface with totally geodesic \textup{(}possibly empty\textup{)} boundary, and let $\mathcal{C}$ be a nonempty geodesic multicurve.
Then there exists a geodesic loop $\gamma$ in $S$ such that $\gamma$ and $\mathcal{C}$ have nonempty intersection.
\end{lem}

\begin{proof}
Choose any loop $c \in \mathcal{C}$.  Cutting $S$ along $c$ produces a surface each of whose components has negative Euler characteristic since each of its components is a union of blocks of $S$ joined along circles.
For any such simple closed curve $c$ in a compact surface $S$, we claim that there exists another closed curve $\gamma$ whose geometric intersection number with $c$ is nonzero.  We leave the proof of this claim as an easy exercise for the reader (using, for example, the techniques in Sections 1.2.4 and~1.3 of \cite{FarbMargalit12}).
After applying a homotopy, we can assume that $\gamma$ is a geodesic loop.
\end{proof}

\begin{lem}
\label{lem:LoopsAreQuasigeodesic}
Let $S$ be a compact hyperbolic surface with totally geodesic \textup{(}possibly empty\textup{)} boundary, and let $\mathcal{C}$ be a nonempty geodesic multicurve.
Let $\mathcal{L}$ be the family of lines that are lifts of loops of $\mathcal{C}$ or boundary loops of $S$.
For any $\kappa > 0$, there exist numbers $\mu$, $L$, and $C$ such that the following holds.

Consider a piecewise geodesic $c = \alpha_1 \beta_1 \cdots \alpha_n \beta_n$ in the universal cover $\tilde S$.
Suppose each segment $\alpha_j$ is contained in a line of $\mathcal{L}$, and each segment $\beta_j$ meets the lines of $\mathcal{L}$ only at its endpoints.
If $\abs{\beta_j} \le \kappa$ and $\abs{\alpha_j} \ge \mu$ for all $j$, then $c$ is an $(L,C)$--quasigeodesic.
\end{lem}

\begin{proof}
In order to show that the piecewise geodesic $c$ in the $\delta$--hyperbolic space $\tilde{S}$ is uniformly quasigeodesic, we will 
show that there exists a constant $D$ such that whenever $\mu$ is sufficiently large, $c$ satisfies the hypotheses of Theorem~\ref{thm:Quasigeodesic}.  Thus $c$ is $(L,C)$--quasigeodesic for constants $L$ and $C$ not depending on $c$.

Let $Q=4\delta + \kappa$. 
Let $D = D(Q)$ be the constant given by applying Lemma~\ref{lem:DeviationTree} to the hyperbolic surface $S$ and the multicurve $\mathcal{C}$.
Let $L=L(\delta,\kappa,D)$ and $C=C(\delta,\kappa,D)$ be the constants given by Theorem~\ref{thm:Quasigeodesic}. We define $\mu = 11\kappa + 25\delta + 2D$.
It follows that $c$ satisfies Conditions (\ref{item:quasigeodesic:Short}) and~(\ref{item:quasigeodesic:Long}) of Theorem~\ref{thm:Quasigeodesic}.

In order to to verify Condition~(\ref{item:quasigeodesic:deviation}) of Theorem~\ref{thm:Quasigeodesic}, we need to show that $\dev_{Q}\bigl (\alpha_{j},\alpha_{j+1} \bigr ) \le D$.
Let $u\in \alpha_j$ and $v \in \alpha_{j+1}$ such that $d(u,v)\le Q$. Let $u'$ and $v'$ be the initial and terminal points of $\beta_j$. We need to show that $d \bigl( u,u' \bigr) \le D$ and $d \bigl( v,v' \bigr) \le D$.
It is easy to see the elements $u$, $v$, $u'$, and $v'$ belong to $\mathcal{N}_{Q} \bigl(\alpha_{j} \bigr) \cap \mathcal{N}_{Q} \bigl( \alpha_{j+1} \bigr)$ because $\abs{\beta_j} \le \kappa \le Q$.
We have $\alpha_{j+1}$ and $\alpha_j$ do not belong to the same line of $\mathcal{L}$ because any two distinct geodesics in $\tilde{S}$ intersect at most one point. Hence
\[
   \diam \bigl( \mathcal{N}_{Q} (\alpha_{j}) \cap \mathcal{N}_{Q}(\alpha_{j+1}) \bigr)
   \leq D
\]
by Lemma~\ref{lem:DeviationTree}.
It follows that $d \bigl( u,u' \bigr) \le D$ and $d \bigl( v,v' \bigr) \le D$. It follows immediately from the definition of the deviation that $\dev_{Q} \bigl( \alpha_{j}, \alpha_{j+1} \bigr) \le D$.
\end{proof}

%%%%%%%%%%%%%%%%%%%%%%%%%%%%%%%%%%%%%%%%%%%%%%%%%%%
\section{Graph manifolds and horizontal surfaces}
\label{sec:GraphManifolds}
%%%%%%%%%%%%%%%%%%%%%%%%%%%%%%%%%%%%%%%%%%%%%%%%%%%

In this section, we review background about graph manifolds and horizontal surfaces. In addition, we discuss a convenient metric for a graph manifold that will be used in next sections. We refer the reader to \cite{RW98}, \cite{BS04} and \cite{KapovichLeeb98} for more details.

\begin{defn}
A \emph{graph manifold} is a compact, irreducible, connected, orientable 3-manifold $N$ that can be decomposed along embedded incompressible tori $\mathcal{T}$ into finitely many Seifert manifolds. We specifically exclude Sol and Seifert manifolds from the class of graph manifolds. Up to isotopy, each graph manifold has a unique minimal collection of tori $\mathcal{T}$ as above \cite{JS79,Johannson79}. This minimal collection is the \emph{JSJ decomposition} of $N$, and each torus of $\mathcal{T}$ is a \emph{JSJ torus}.
\end{defn}

Throughout this paper, a graph consists of a set $\mathcal{V}$ of vertices and a set of $\mathcal{E}$ of edges, each edge being associated to an unordered pair of vertices by a function $ends \colon ends(e) = \{v,v'\}$ where $v,v' \in \mathcal{V}$. In this case we call $v$ and $v'$ the \emph{endpoints} of the edge $e$ and we also say $v$ and $v'$ are adjacent.

\begin{defn}
\label{defn:SimpleGraphManifold}
A \emph{simple graph manifold} $N$ is a graph manifold with the following properties:
\begin{enumerate}
\item Each Seifert component is a trivial circle bundle over an orientable surface of genus at least 2.
\item The intersection numbers of fibers of adjacent Seifert components have absolute value 1.
\end{enumerate}
\end{defn}

\begin{thm}[\cite{KapovichLeeb98}]
\label{thm:KLsimple}
Any graph manifold $N$ has a finite cover $\hat{N}$ that is a simple graph manifold.
\end{thm}

\begin{defn}
Let $M$ be a Seifert manifold with boundary.
A \emph{horizontal surface} in $M$ is an immersion $g\colon B\looparrowright M$ where $B$ is a compact surface with boundary such that the image $g(B)$ is transverse to the Seifert fibration.
We also require that $g$ is \emph{properly immersed}, ie, $g(B) \cap \boundary M = g(\boundary B)$.

A \emph{horizontal surface} in a graph manifold $N$ is a properly immersed surface $g\colon S\looparrowright{N}$ such that for each Seifert component $M$, the intersection $g(S)\cap{M}$ is a horizontal surface in $M$.
A horizontal surface $g$ in a graph manifold $N$ is always $\pi_{1}$--injective and lifts to an embedding of $S$ in the cover of $N$ corresponding to the subgroup $g_*\bigl( \pi_1(S) \bigr)$ by \cite{RW98}.
Consequently, $g$ also lifts to an embedding $\tilde{S} \to \tilde{N}$ of universal covers.
\end{defn}

\begin{defn}
\label{def:Separable}
A horizontal surface $g\colon{S}\looparrowright{N}$ in a graph manifold $N$ is \emph{virtually embedded} if $g$ lifts to an embedding of $S$ in some finite cover of $N$.
By a theorem of Scott \cite{Scott78}, a horizontal surface is virtually embedded if $g_* \bigl( \pi_1(S) \bigr)$ is a \emph{separable} subgroup of $\pi_1(N)$, \emph{i.e.}, it is equal to an intersection of finite index subgroups.
Przytycki--Wise have shown that the converse holds as well \cite{PrzytyckiWise14GraphSpecial}.
\end{defn}

The following result about separability allows one to pass to finite covers in the study of horizontal surfaces, as explained in Corollary~\ref{cor:ScottFiniteCover}.

\begin{prop}[\cite{Scott78}, Lemma~1.1]
\label{prop:ScottFiniteIndex}
Let $G_0$ be a finite index subgroup of $G$.  A subgroup $H \le G$ is separable in $G$ if and only if $H \cap G_0$ is separable in $G_0$.
\end{prop}

\begin{cor}
\label{cor:ScottFiniteCover}
Let $q\colon (\hat{N},\hat{x}_0) \to (N,x_0)$ be a finite covering of graph manifolds.
Let $g\colon (S,s_0) \looparrowright (N,x_0)$ be a horizontal surface.
Let $p\colon (\hat{S},\hat{s}_0) \to (S,s_0)$ be the finite cover corresponding to the subgroup $g_*^{-1} q_* \pi_1(\hat{N},\hat{x}_0)$.
Then $g$ lifts to a horizontal surface $\hat{g} \colon (\hat{S},\hat{s}_0) \to (\hat{N},\hat{x}_0)$.
Furthermore $g$ is a virtual embedding if and only if $\hat{g}$ is a virtual embedding.
\end{cor}

\begin{defn}
\label{defn:Blocks}
Suppose $g \colon S \looparrowright N$ is a horizontal surface in a graph manifold $N$ with JSJ decomposition $\mathcal{T}$. 
Let $\mathcal{T}_{g}$ denote the collection of components of $g^{-1}(\mathcal{T})$ in $S$.
After applying a homotopy to $g$, we may assume that the image $g(c)$ of each curve $c\in{\mathcal{T}_{g}}$ is a multiple of a simple closed curve on the corresponding JSJ torus.
The connected components of the splitting $S|g^{-1}(\mathcal{T})$ are the \emph{blocks} $B$ of $S$. 
\end{defn}

\begin{rem}
\label{rem:MeetsEachFiber}
If $g\colon S \looparrowright N$ is a horizontal surface then $\mathcal{T}_g$ is always nonempty.
Indeed, $g(S)$ has nonempty intersection with each JSJ torus of $N$ because a properly immersed horizontal surface in a connected graph manifold must intersect every fiber of every Seifert component (see Lemma~\ref{lem:UniqueIntersection} for details).
\end{rem}

In \cite{RW98}, Rubinstein--Wang introduced the dilation of a horizontal surface, and proved that dilation is the obstruction to a surface being virtually embedded (see Theorem~\ref{thm:RubinsteinWang}).

\begin{defn}[Dilation]
\label{defn:DilationSlopes}
Let $g \colon (S,s_{0}) \looparrowright (N,x_{0})$ be a horizontal surface in a simple graph manifold $N$. Choose an orientation for the graph manifold $N$, an orientation for the fiber of each Seifert component, and an orientation for each curve $c \in \mathcal{T}_{g}$.

The \emph{dilation} of a horizontal surface $S$ in $N$ is a homomorphism $w \colon \pi_1(S,s_0) \to \Q_{+}^{*}$ defined as follows. Choose $[\gamma] \in \pi_1(S,s_0)$ such that $\gamma$ is transverse to $\mathcal{T}_g$.  
In the trivial case that $\gamma$ is disjoint from the curves of the collection $\mathcal{T}_g$, we set $w(\gamma) = 1$.  Let us assume now that this intersection is nonempty.
Then $\mathcal{T}_g$ subdivides $\gamma$ into a concatenation $\gamma_1\cdots\gamma_m$ with the following properties.
Each path $\gamma_i$ starts on a circle $c_i \in \mathcal{T}_g$ and ends on the circle $c_{i+1}$.  The image $g(\gamma_i)$ of this path in $N$ lies in a Seifert component $M_i$.  The image of the circle $g(c_i)$ in $N$ lies in a JSJ torus $T_i$ obtained by gluing a boundary torus $\overleftarrow{T_i}$ of $M_{i-1}$ to a boundary torus $\overrightarrow{T_i}$ of $M_i$.
Let $\overleftarrow{f_i}$ and $\overrightarrow{f_i}$ be fibers of $M_{i-1}$ and $M_i$ in the torus $T_i$.
By Definition~\ref{defn:SimpleGraphManifold}, the $1$--cycles $[\overleftarrow{f_i}]$ and $[\overrightarrow{f_i}]$ generate the integral homology group $H_1(T_i) \cong \Z^2$, so there exist integers $a_{i}$ and $b_{i}$ such that
\[
   \bigl[g(c_{i})\bigr]=a_{i}[\overleftarrow{f_i}]+b_{i}[\overrightarrow{f_i}]
   \quad \text{in} \quad
   H_1(T_i).
\]
Since the immersion is horizontal, these coefficients $a_i$ and $b_i$ must be nonzero.
The dilation $w(\gamma)$ is the rational number
$\prod_{i=1}^{m} \bigl| b_{i}/a_{i} \bigr|$. 
Note that $w(\gamma)$ depends only on the homotopy class of $\gamma$, since crossings of $\gamma$ with a curve $c \in \mathcal{T}_g$ in opposite directions contribute terms to the dilation that cancel each other. For the rest of this paper we write $w_{\gamma}$ instead of $w(\gamma)$.
\end{defn}

The following result is a special case of \cite[Theorem~2.3]{RW98}.

\begin{thm}
\label{thm:RubinsteinWang}
A horizontal surface $g\colon{S}\looparrowright{N}$ in a simple graph manifold is virtually embedded if and only if the dilation $w$ is the trivial homomorphism.
\end{thm}

\begin{rem}
\label{rem:SurfaceAnnulus}
Let $g\colon{S}\looparrowright{N}$ be a horizontal surface in a simple graph manifold $M$, then each block $B$ is a connected surface with non-empty boundary and negative Euler characteristic. Indeed, the immersion $g\colon{S}\looparrowright{N}$ maps $B$ to the corresponding Seifert component $M$ with base surface $F$. The composition of $g|_{B}$ with the projection of $M$ to $F$ yields a finite covering map from $B$ to $F$.  Since $\chi(F)<0$, it follows that $\chi(B)<0$ as well.

We note that the collection $\mathcal{T}_{g}$ is always a non-empty multicurve. Indeed, $\mathcal{T}_g$ is nonempty by Remark~\ref{rem:MeetsEachFiber}.  Since the blocks of $S$ have negative Euler characteristic, it follows that $\mathcal{T}_{g}$ is a multicurve.
\end{rem}

\begin{rem}
\label{rem:MetricOnGraphManifold}
We now are going to describe a convenient metric on a simple graph manifold $N$ introduced by Kapovich--Leeb \cite{KapovichLeeb98}. For each Seifert component $M_i = F_{i} \times S^1$ of $N$, we choose a hyperbolic metric on the base surface $F_i$ so that all boundary components are totally geodesic of unit length, and then equip each Seifert component $M_i = F_i \times S^1$ with the product metric $d_i$ such that the fibers have length one. Metrics $d_i$ on $M_i$ induce the product metrics on $\tilde{M}_i$ which by abuse of notations is also denoted by $d_i$.

Let $M_i$ and $M_j$ be adjacent Seifert components in the simple graph manifold $N$, and let $T \subset M_i \cap M_j$ be a JSJ torus. Each metric space $(\tilde{T},d_{i})$ and $(\tilde{T}, d_{j})$ is a Euclidean plane. After applying a homotopy to the gluing map, we may assume that at each JSJ torus $T$, the gluing map $\phi$ from the boundary torus $\overleftarrow{T} \subset M_i$ to the boundary torus $\overrightarrow{T} \subset M_j$ is affine in the sense that the identity map $(\tilde{T},d_{i}) \to (\tilde{T}, d_{j})$ is affine.
We now have a product metric on each Seifert component $M_i = F_i \times S^1$. These metrics may not agree on the JSJ tori but the gluing maps are bilipschitz (since they are affine). 
The product metrics on the Seifert components induce a length metric on the graph manifold $N$ denoted by $d$ (see Section~3.1 of \cite{BBI01} for details). Moreover, there exists a positive constant $K$ such that on each Seifert component $M_i = F_i \times S^1$ we have
\[
   \frac{1}{K} \, d_i(x,y)
   \leq d(x,y)
   \leq K \, d_i(x,y)
\]
for all $x$ and $y$ in $M_i$. (See Lemma~1.8 of \cite{Paulik05} for a detailed proof of the last claim.)
Metric $d$ on $N$ induces metric on $\tilde{N}$ , which is also denoted by $d$ (by abuse of notations). Then for all $x$ and $y$ in $\tilde{M}_i$ we have
\[
   \frac{1}{K} \, d_{i}(x,y)
   \leq d(x,y)
   \leq K \, d_{i}(x,y)
\]
\end{rem}

The following remark introduces certain invariants of a horizontal surface $g$ that will be used in the proof of Theorem~\ref{thm:UpperBound}.

\begin{rem}\hfill
\label{rem:constants:setting up}
\begin{enumerate}
\item
\label{item:lowerbound_JSJplanes}
Since $N$ is compact, there exists a positive lower bound $\rho$ for the distance between any two distinct JSJ planes in $\tilde{N}$.

\item We recall that for each curve $c_i$ in $\mathcal{T}_g$, there exist non-zero integers $a_{i}$ and $b_{i}$ such that
\[
   \bigl[g(c_{i})\bigr]=a_{i}[\overleftarrow{f_i}]+b_{i}[\overrightarrow{f_i}]
   \quad \text{in} \quad
   H_1(T_i).
\]
The \emph{governor} of a horizontal surface $g \colon S \looparrowright N$ in a graph manifold is the quantity 
$\epsilon = \epsilon(g) = \max \bigset{\bigabs{a_{i}/b_{i}},\bigabs{b_{i}/a_{i}}}{c_{i} \in \mathcal{T}_g}$.
(We use the term ``governor'' here in the sense of a device used to limit the top speed of a vehicle or engine.  In this context the governor limits the rate of growth of the products used in calculating the dilation of curves in the surface.)
\end{enumerate}
\end{rem}

\begin{prop}
\label{prop:upperlambda}
For each $[\gamma] \in \pi_1(S,s_0)$ as in Definition~\ref{defn:DilationSlopes}, we define 
\[
\Lambda_{\gamma} =\max \bigset{\prod_{i=j}^{k} \bigl| b_{i}/a_{i} \bigr|}{1 \le j \le k \le m}
\]
If the horizontal surface $g$ is a virtual embedding, then there exists a positive constant $\Lambda$ such that $\Lambda_{\gamma} \le \Lambda$ for all $[\gamma] \in \pi_1(S,s_0)$. 
\end{prop}

\begin{proof}
Let $\Gamma_g$ be the graph dual to $\mathcal{T}_g$, and let $n$ be the number of vertices in $\Gamma_g$.
Each oriented edge $e$ of $\Gamma_g$ is dual to a curve $c \in \mathcal{T}_g$ and determines a slope $\bigabs{b/a}$ as described in Definition~\ref{defn:DilationSlopes} and Remark~\ref{rem:constants:setting up}.
By Theorem~\ref{thm:RubinsteinWang} the dilation is trivial for any loop in $S$.
Therefore for each cycle $e_1 \cdots e_m$ in $\Gamma_g$ the corresponding product of slopes $\prod_{i=1}^m \bigabs{b_i/a_i}$ is trivial. It follows that for any edge path $e_1 \cdots e_m$ the value of the product $\prod_{i=1}^m \bigabs{b_i/a_i}$
 depends only on the endpoints of the path.
Each slope $\bigabs{b_i/a_i}$ is bounded above by the governor $\epsilon$ of $g$.
Since any two vertices of $\Gamma_g$ are joined by a path of length less than $n$, the result follows, using $\Lambda = \epsilon^n$.
\end{proof}

\begin{rem}
\label{rem:straightline in torus}
Let $g\colon{S}\looparrowright{N}$ be a horizontal surface in a simple graph manifold. Equip $N$ with the metric $d$ described in Remark~\ref{rem:MetricOnGraphManifold}. By \cite[Lemma~3.1]{Neumann01} the surface $g$ can be homotoped to another horizontal surface $g' \colon S \looparrowright N$ such that the following holds: For each curve $c$ in $g'^{-1}(\mathcal{T})$, let $T$ be the JSJ torus in $N$ such that $g'(c) \subset T$. Then $g'(c)$ is $\emph{straight}$ in $T$ in the sense that lifts of $g'(c)$ to $\tilde{N}$ are straight lines in the JSJ planes containing it.
\end{rem}

\begin{rem}
\label{rem:nonempty intersection}
Let $f \colon S^1 \looparrowright T =  S^1 \times S^1$ be a horizontal immersion. Then every fiber $\{x\} \times S^1$ in the torus $T$ has non-empty intersection with $f(S^1)$. Indeed, this follows from the fact the composition of $f$ with the natural projection of $T$ to the first factor $S^1$ is a finite covering map.
\end{rem}

%%%%%%%%%%%%%%%%%%%%%%%%%%%%%%%%%%%%%%%%%%%%%%%%%%%%%%%%%
\section{A lower bound for distortion}
\label{sec:LowerBound}
%%%%%%%%%%%%%%%%%%%%%%%%%%%%%%%%%%%%%%%%%%%%%%%%%%%%%%%%%

The main theorem in this section is the following.

\begin{thm} 
\label{thm:lower}
Let $g\colon (S,s_0) \looparrowright (N,x_{0})$ be a  horizontal surface in a graph manifold $M$.  Let $G=\pi_1(N,x_{0})$ and $H=g_{*}(\pi_1(S,s_{0}))$. Then the distortion $\Delta_{H}^{G}$ is at least quadratic. Furthermore, if the horizontal surface $g$ is not virtually embedded, then the distortion $\Delta_{H}^{G}$ is at least exponential.
\end{thm}

To see the proof of Theorem~\ref{thm:lower}, we need several lemmas. For the rest of this section we fix a horizontal surface $g\colon S \looparrowright N$ in a simple graph manifold $N$. We equip $N$ with the metric $d$ described in Remark~\ref{rem:MetricOnGraphManifold} and equip $S$ with a hyperbolic metric $d_S$ such that the boundary (if nonempty) is totally geodesic and the simple closed curves of $\mathcal{T}_g$ are geodesics.

\begin{lem} 
\label{lem:ConstructSequence}
Let $\gamma$ be any geodesic loop in $S$ such that $\gamma$ and $\mathcal{T}_g$ have nonempty intersection and such that $w_\gamma \ge 1$.
There exists a positive number $A = A(\gamma)$ such that for all $\mu >0$ the following holds:
Let $\{c_1,\dots,c_m\}$ be the sequence of curves of $\mathcal{T}_{g}$ crossed by $\gamma$. The image of the circle $g(c_i)$ in $M$ lies in a JSJ torus $T_i$. For each $i= 1,2,\dots,m$, let $a_{i}$ and $b_{i}$ be the integers such that
\[
   \bigl[ g(c_i) \bigr] = a_{i}[\overleftarrow{f_i}] + b_{i}[\overrightarrow{f_i}] \quad\text{in} \quad
   H_1(T_i;\Z).
\]
Extend the sequence $a_1,\dots,a_m$ to a periodic sequence $\{a_j\}_{j=1}^\infty$ with $a_{j+m} = a_j$ for all $j>0$, and similarly extend $b_1,\dots,b_m$ to an $m$--periodic sequence $\{b_j\}_{j=1}^\infty$. Then there exists a \textup{(}nonperiodic\textup{)} sequence of integers $\bigl\{ t(j) \bigr\}_{j=1}^\infty$, depending on our choice of the constant $\mu$ and the loop $\gamma$, with the following properties:
\begin{enumerate}
   \item
   \label{item:sequence:bounded below}
   $\bigabs{t(j)}\geq{\mu}$ for all $j$.
   \item
   \label{item:sequence:ApproxUpper}
   $\bigabs{t(j)\,a_j +t(j-1)\,b_{j-1}} \le A$ for all $j > 1$.
   \item
   \label{item:sequence:growing}
   The partial sum $f(n) = \sum_{j=1}^{nm} \bigabs{t(j)}$ satisfies $f(n) \succeq n^2 + w_\gamma^n$.
\end{enumerate}
\end{lem}

\begin{proof}
By the definition of the dilation function, we have $w_{\gamma}=\prod_{i=1}^{m} \bigabs{b_{i}/a_{i}}$.

Let $A =  \max \bigset{1+\abs{a_i}}{i=1,2,\dots ,m}$.
Set $\xi = \min{\bigl\{ \abs{1/a_j} \bigr\}}$, and choose $\lambda \in (0,1]$ so that $\lambda \le \bigabs{b_{j-1}/a_j}$ for all $j>1$.
Starting from an initial value $t(1) \ge \mu/\lambda^{m-1}$,
we recursively construct an infinite sequence $\bigl\{ t(j) \bigr\}$ satisfying (\ref{item:sequence:ApproxUpper}).
Suppose that $t(j-1)$ has been defined for some $j>1$, and we wish to define $t(j)$.  As $A\geq 1+\abs{a_j}$, we have
\[
   1 \leq \frac{A-1}{\abs{a_j}}
   = \frac{A+ \bigabs{t(j-1) \, b_{j-1}}}{\abs{a_j}}
     - \frac{1+ \bigabs{t(j-1) \, b_{j-1}}}{\abs{a_j}}.
\]
It follows that there is an integer $t(j)$ such that 
\begin{equation}
\tag{$\clubsuit$}
\label{eqn:double:inequality}
   \frac{1 + \bigabs{ t(j-1) \, b_{j-1}} }{\abs{a_{j}}}
   \leq \abs{t(j)}
   \leq \frac{A + \bigabs{ t(j-1) \, b_{j-1}}}{\abs{a_{j}}},
\end{equation}
which is equivalent to
\[
   1 \le 
   \bigabs{t(j)\,a_j} - \bigabs{t(j-1)\,b_{j-1}} \le A.
\]
Furthermore, we are free to choose the sign of $t(j)$ so that $t(j)\,a_j$ and $t{(j-1)}\,b_{j-1}$ have opposite signs, which immediately gives (\ref{item:sequence:ApproxUpper}).
By induction, the sequence $\bigl\{ t(j) \bigr\}$ satisfies both \eqref{eqn:double:inequality} and~(\ref{item:sequence:ApproxUpper}) for all $j>1$.

We next show that any sequence $\bigl\{ t(j) \bigr\}$ satisfying \eqref{eqn:double:inequality} also satisfies (\ref{item:sequence:growing}).
Indeed, the first inequality of \eqref{eqn:double:inequality} implies 
\begin{equation}
\tag{$\diamondsuit$}
\label{eqn:SequenceIncreases}
   \bigabs{t(j)}
     \ge \frac{1}{|a_{j}|} + \left| \frac{b_{j-1}}{a_{j}} \right| \, \bigabs{t(j-1)}
     \ge \left| \frac{b_{j-1}}{a_{j}} \right| \, \bigabs{t(j-1)}
   \qquad \text{for all $j>1$.}
\end{equation}
For any $j>m$, we apply \eqref{eqn:SequenceIncreases} iteratively $m$ times (and use that the sequences $\bigl\{ a_j \bigr\}$ and $\bigl\{ b_j \bigr\}$ are $m$--periodic) to get
\begin{equation}
\tag{$\heartsuit$}
\label{eqn:SequenceOneCycle}
   \begin{split}
   \bigabs{t(j)}
     &\geq \frac{1}{|a_{j}|} + \left| \frac{b_{j-1}}{a_{j}} \right| \, \bigabs{t(j-1)} \\
     &\geq \frac{1}{|a_{j}|} + \left| \frac{ b_{j-1} \, b_{j-2} }{ a_j\, a_{j-1} } \right| \, \bigabs{t(j-2)}
     \geq \cdots \\
     &\geq \frac{1}{|a_{j}|} +
     \left| \frac{b_1 \cdots b_m}{a_1 \cdots a_m} \right| \, \bigabs{t(j-m)}  \\
     &\geq \xi + w_{\gamma} \, \bigabs{t(j-m)}.
   \end{split}
\end{equation}
Further applying \eqref{eqn:SequenceOneCycle} iteratively $k$ times (and using that $w_\gamma \ge 1$) gives
\begin{equation}
\tag{$\spadesuit$}
\label{eqn:SequenceUltimateGrowth}
   \begin{split}
   \bigabs{t(j)}
     &\geq \xi + w_{\gamma}\,\bigabs{t(j-m)}  \\
     &\geq \xi + w_{\gamma}\bigl( \xi + w_{\gamma} \, \bigabs{t(j-2m)} \bigr) \\
     &\geq 2\xi + w_\gamma^2 \bigabs{t(j-2m)} \ge \cdots \\
     &\geq \xi k + w_{\gamma}^{k} \, \bigabs{ t(j-km) }  \qquad \text{for all $j>km$}
   \end{split}
\end{equation}
The inequality \eqref{eqn:SequenceUltimateGrowth} can be rewritten in the form
\[
   \bigabs{t(km + 1)}
   \geq \xi k + \bigabs{t(1)} \, w_{\gamma}^{k}
   \qquad \text{for all $k>0$.}
\]
Finally for each positive $n$ we observe that
\[
   \sum_{j=1}^{nm} \bigabs{t(j)}
     \geq \sum_{k=1}^{n-1} \bigabs{t(km + 1)}
     \geq \xi \sum_{k=1}^{n-1} k + \bigabs{t(1)} \sum_{k=1}^{n-1} w_\gamma^k
\]
which implies (\ref{item:sequence:growing}) as desired.

In order to establish (\ref{item:sequence:bounded below}), recall that $\gamma$ satisfies $w_{\gamma} \ge 1$. Therefore \eqref{eqn:SequenceOneCycle} implies that $\bigabs{t(j)} \ge \abs{t(j-m)}$ for any $j>m$.
In particular, it follows that the terms of the sequence $\bigl\{ t(j) \bigr\}$ have absolute value bounded below by the absolute values of the first $m$ terms: $t(1),\dots,t(m)$.

Using our choice of $t(1) \ge \mu/\lambda^{m-1}$ and the fact $\lambda \le \bigabs{b_{j-1} / a_j}$, the inequality \eqref{eqn:SequenceIncreases} shows that for all $i = 1,\dots,m$ we have
\[
   \bigabs{t(i)}
     \ge \lambda^{i-1} \bigabs{t(1)}
     \ge \frac{\mu}{\lambda^{m-i}}.
\]
As $\lambda \le 1$, we conclude that $\bigabs{t(i)} \ge \mu$, completing the proof of (\ref{item:sequence:bounded below}).
\end{proof}

\begin{defn}[spirals]
Let $\gamma$ be a closed curve in $S$ satisfying the conclusion of Lemma~\ref{lem:ExistenceOfLoop}. Let $\{c_1,\dots,c_m\}$ be the sequence of curves of $\mathcal{C}$ crossed by $\gamma$. Let $\{\gamma_1,\dots,\gamma_m\}$ be the sequence of subpaths of $\gamma$ introduced in Definition~\ref{defn:DilationSlopes}.  Extend the finite sequences of curves $c_1,\dots,c_m$ and $\gamma_1,\dots,\gamma_m$ to $m$--periodic infinite sequences $\{c_{j}\}_{j=1}^\infty$ and $\{\gamma_{j}\}_{j=1}^\infty$. 
Let $\bigl\{ t(j) \bigr\}$ be a sequence of integers. We denote $\alpha_j = c_{j}^{t(j)}$.

Choose the basepoint $s_0 \in S$ to be a point of intersection between $\gamma$ and one of the curves of the family $\mathcal{C}$.
For each $n\in \N$, we define a \emph{spiral} loop $\sigma_n$ in $S$ based at $s_0$ as a concatenation:
\[
   \sigma_n
    = \alpha_1 \gamma_1 \cdots \alpha_{nm} \gamma_{nm}
\]
and a \emph{double spiral} loop $\rho_n$ of $\sigma_n$ in $S$ based at $s_0$ as $\rho_n=\sigma_n \alpha'_{nm + 1} \sigma_n^{-1}$ where $\alpha'_{nm + 1} = c_{nm + 1}^{t(1)}$.
\end{defn}

\begin{lem}
\label{lem:LinearInSurface}
Let $\bigl\{ t(j) \bigr\}_{j=1}^\infty$ be the sequence of integers given by Lemma~\ref{lem:ConstructSequence}.
Let $\rho_n$ be the double spiral loop of $\sigma_n$ coresponding to the curve $\gamma$ and the sequence $\bigl\{ t(j) \bigr\}$.
Let $\tilde{\rho}_n$ be the lift of $\rho_n$ in $\tilde{S}$.
Then the distance in $\tilde{N}$ between the endpoints of $\tilde{g}(\tilde{\rho_n})$ is bounded above by a linear function of $n$.
\end{lem}

\begin{proof}
First we describe informally the idea of the proof, which is illustrated in Figure~\ref{fig:spiral}. 
The lift $\tilde{\sigma}_n$ of $\sigma_n$ is the spiral-shaped curve running around the outside of the left-hand diagram. The path $\tilde{\sigma}_n$ alternates between long segments $\tilde{\alpha}_j$ belonging to JSJ planes and short segments $\tilde{\gamma}_j$ belonging to Seifert components. Each long segment $\tilde{\alpha}_j$ is one side of a large triangle in the JSJ plane whose other two sides are fibers of the adjacent blocks, which meet in a corner $y_{j}$ opposite to $\tilde{\alpha}_j$.
Connecting each pair of adjacent corners $y_{j}$ 
produces a thin trapezoid that interpolates between two adjacent JSJ triangles.
We will see that the sequence of exponents $\bigl\{ t(j) \}$ in the construction of $\tilde{\sigma}_n$ was chosen carefully to ensure that the distances between adjacent corners $y_{j}$ are all short, \emph{i.e.}, bounded above.
Thus the path running around the inside of the spiral has at most a linear length---except for its last segment, which is a long side of a large JSJ triangle.
The double spiral gives rise to a diagram similar to the spiral diagram, expect that it has been doubled along this long triangular side, so that the long side no longer appears on the boundary of the diagram but rather appears in its interior (see Figure~\ref{fig:doublespiral}).

\usetikzlibrary{arrows}
\begin{figure}

\definecolor{wqwqwq}{rgb}{0.3764705882352941,0.3764705882352941,0.3764705882352941}
\begin{tikzpicture}[line cap=round,line join=round,>=triangle 45,x=1.0cm,y=1.0cm]
\clip(-3.8088900632617935,-0.735505967949124) rectangle (7.724379762091707,5.267425231472002);
\fill[color=wqwqwq,fill=wqwqwq,fill opacity=0.2] (3.54,4.6) -- (4.26,1.9) -- (6.24,3.46) -- cycle;
\fill[color=wqwqwq,fill=wqwqwq,fill opacity=0.2] (6.523,2.768) -- (4.0546,0.7836) -- (6.1116,-0.378) -- cycle;
\fill[line width=1.5pt,color=wqwqwq,fill=wqwqwq,fill opacity=0.2] (-2.2146,3.6178) -- (1.3484,4.7452) -- (0.2747069408740359,1.7256982005141386) -- cycle;
\fill[line width=0.5pt,color=wqwqwq,fill=wqwqwq,fill opacity=0.2] (-2.48,0.72) -- (-2.433449822454952,3.166017427960563) -- (-0.059,1.5294) -- cycle;
\fill[color=wqwqwq,fill=wqwqwq,fill opacity=0.2] (-0.98,-0.58) -- (-2.34,0.36) -- (-0.32,1.02) -- cycle;
\fill[line width=0.5pt,fill=wqwqwq,fill opacity=0.2] (0.52,-0.38) -- (-0.6,-0.64) -- (-0.02,0.76) -- cycle;
\fill[color=wqwqwq,fill=wqwqwq,fill opacity=0.2] (1.5104,1.5348) -- (1.5378,0.778) -- (0.5758,1.1316) -- cycle;
\fill[line width=0.5pt,color=wqwqwq,fill=wqwqwq,fill opacity=0.2] (1.361,0.4222) -- (0.9086,-0.136) -- (0.4,0.82) -- cycle;
\fill[color=wqwqwq,fill=wqwqwq,fill opacity=0.2] (1.3872,1.9758) -- (0.8844,2.284) -- (0.5726,1.6126) -- cycle;
\draw [line width=1.5pt] (1.3872,1.9758)-- (1.5104,1.5348);
\draw [line width=1.5pt] (1.4711201283437643,1.675404086042209) -- (1.354377556231702,1.7289216665028249);
\draw [line width=1.5pt] (1.4711201283437643,1.675404086042209) -- (1.5432224437682995,1.7816783334971755);
\draw [line width=1.5pt] (1.5378,0.778)-- (1.361,0.4222);
\draw [line width=1.5pt] (1.4124852131900894,0.5258110794854864) -- (1.3616040030283019,0.6437265662298939);
\draw [line width=1.5pt] (1.4124852131900894,0.5258110794854864) -- (1.537195996971698,0.5564734337701064);
\draw [line width=1.5pt] (0.4,0.82)-- (0.5758,1.1316);
\draw [line width=1.5pt] (-0.02,0.76)-- (0.4,0.82);
\draw [line width=1.5pt] (-0.6,-0.64)-- (-0.98,-0.58);
\draw [line width=1.5pt] (-0.871939955206372,-0.5970621123358362) -- (-0.7747097691241689,-0.5131618711197419);
\draw [line width=1.5pt] (-0.871939955206372,-0.5970621123358362) -- (-0.8052902308758297,-0.706838128880259);
\draw [line width=1.5pt] (-0.32,1.02)-- (-0.02,0.76);
\draw [line width=1.5pt] (-2.34,0.36)-- (-2.48,0.72);
\draw [line width=1.5pt] (-2.4400667583683493,0.6173145215186135) -- (-2.3186282927507293,0.5755334417080498);
\draw [line width=1.5pt] (-2.4400667583683493,0.6173145215186135) -- (-2.5013717072492705,0.5044665582919508);
\draw [line width=1.5pt] (-0.059,1.5294)-- (-0.32,1.02);
\draw [line width=1.5pt] (0.281,1.7294)-- (-0.059,1.5294);
\draw [line width=1.5pt] (6.24,3.46)-- (6.523,2.768);
\draw [line width=1.5pt] (6.412900876665428,3.0372176443375443) -- (6.290757216035282,3.0768898730317686);
\draw [line width=1.5pt] (6.412900876665428,3.0372176443375443) -- (6.47224278396472,3.151110126968228);
\draw [line width=1.5pt] (4.0546,0.7836)-- (4.26,1.9);
\draw [line width=1.5pt] (-2.433449822454952,3.1660174279605613)-- (-2.2146,3.6178);
\draw [line width=1.5pt] (-2.287860070822,3.4665655925306154) -- (-2.2357940547588986,3.3491684480465373);
\draw [line width=1.5pt] (-2.287860070822,3.4665655925306154) -- (-2.4122557676960525,3.434648979914025);
\draw [line width=0.5pt] (0.52,-0.38)-- (-0.6,-0.64);
\draw [line width=0.5pt] (-0.6,-0.64)-- (-0.02,0.76);
\draw [line width=0.5pt] (-0.02,0.76)-- (0.52,-0.38);
\draw [line width=0.5pt] (1.5378,0.778)-- (0.5758,1.1316);
\draw [line width=0.5pt] (0.5758,1.1316)-- (1.5104,1.5348);
\draw [line width=0.5pt] (0.5726,1.6126)-- (1.3872,1.9758);
\draw [line width=1.5pt] (1.7419545938239018,4.46690328536827)-- (0.2747069408740359,1.7256982005141386);
\draw [line width=1.5pt] (-2.2146,3.6178)-- (1.3484,4.7452);
\draw [line width=1.5pt] (-0.3540097787390659,4.206525628809873) -- (-0.4035242568610585,4.088029738509807);
\draw [line width=1.5pt] (-0.3540097787390659,4.206525628809873) -- (-0.4626757431389393,4.274970261490194);
\draw [line width=0.5pt] (-2.2146,3.6178)-- (0.2747069408740359,1.7256982005141386);
\draw [line width=0.5pt] (1.3484,4.7452)-- (0.2747069408740359,1.7256982005141386);
\draw [line width=1.5pt] (-2.48,0.72)-- (-2.433449822454952,3.166017427960563);
\draw [line width=1.5pt] (-2.4551464782496017,2.0259487751435534) -- (-2.358704838943609,1.9411432931882486);
\draw [line width=1.5pt] (-2.4551464782496017,2.0259487751435534) -- (-2.5547449835113425,1.9448741347723144);
\draw [line width=0.5pt] (-2.433449822454952,3.166017427960563)-- (-0.059,1.5294);
\draw [line width=0.5pt] (-0.059,1.5294)-- (-2.48,0.72);
\draw [line width=1.5pt] (-0.98,-0.58)-- (-2.34,0.36);
\draw [line width=1.5pt] (-1.7282411262224517,-0.06283333922859999) -- (-1.6042575827247092,-0.029351396282558093);
\draw [line width=1.5pt] (-1.7282411262224517,-0.06283333922859999) -- (-1.7157424172752918,-0.19064860371744283);
\draw [line width=0.5pt] (-2.34,0.36)-- (-0.32,1.02);
\draw [line width=0.5pt] (-0.32,1.02)-- (-0.98,-0.58);
\draw [line width=0.5pt] (-0.02,0.76)-- (0.52,-0.38);
\draw [line width=1.5pt] (1.361,0.4222)-- (0.9086,-0.136);
\draw [line width=1.5pt] (1.0825682850670668,0.07865317578345932) -- (1.0586355713804514,0.2048283903752843);
\draw [line width=1.5pt] (1.0825682850670668,0.07865317578345932) -- (1.2109644286195482,0.08137160962471551);
\draw [line width=0.5pt] (0.9086,-0.136)-- (0.4,0.82);
\draw [line width=0.5pt] (0.4,0.82)-- (1.361,0.4222);
\draw [line width=0.5pt] (1.5378,0.778)-- (0.5758,1.1316);
\draw [line width=1.5pt] (0.8844,2.284)-- (0.5726,1.6126);
\draw [line width=1.5pt] (0.9086,-0.136)-- (0.52,-0.38);
\draw [line width=1.5pt] (0.6440457950565651,-0.30211226455532264) -- (0.6621673237073469,-0.17497230324866883);
\draw [line width=1.5pt] (0.6440457950565651,-0.30211226455532264) -- (0.7664326762926537,-0.34102769675133277);
\draw [line width=1.5pt] (0.5726,1.6126)-- (0.5758,1.1316);
\draw [line width=1.5pt] (3.54,4.6)-- (6.24,3.46);
\draw [line width=1.5pt] (4.96642233050395,3.997732793787221) -- (4.851866029021262,3.9396827003135138);
\draw [line width=1.5pt] (4.96642233050395,3.997732793787221) -- (4.928133970978739,4.120317299686485);
\draw [line width=0.5pt] (3.54,4.6)-- (4.26,1.9);
\draw [line width=0.5pt] (3.911658777617893,3.206279583932892) -- (3.8417061119105185,3.2344549631761397);
\draw [line width=0.5pt] (3.911658777617893,3.206279583932892) -- (3.958293888089475,3.265545036823861);
\draw [line width=0.5pt] (4.26,1.9)-- (6.24,3.46);
\draw [line width=0.5pt] (5.285542104015833,2.7080028698306533) -- (5.287337159774207,2.632610527978892);
\draw [line width=0.5pt] (5.285542104015833,2.7080028698306533) -- (5.2126628402257955,2.7273894720211063);
\draw [line width=1.5pt] (6.523,2.768)-- (6.1116,-0.378);
\draw [line width=1.5pt] (6.306543608548767,1.1127452418435209) -- (6.220089831269616,1.207712098987819);
\draw [line width=1.5pt] (6.306543608548767,1.1127452418435209) -- (6.4145101687303825,1.1822879010121798);
\draw [line width=0.5pt] (6.523,2.768)-- (4.0546,0.7836);
\draw [line width=0.5pt] (5.253534634193321,1.747449411802473) -- (5.250999215736631,1.822820487742239);
\draw [line width=0.5pt] (5.253534634193321,1.747449411802473) -- (5.326600784263368,1.7287795122577605);
\draw [line width=0.5pt] (4.0546,0.7836)-- (6.1116,-0.378);
\draw [line width=0.5pt] (5.122500053428388,0.18055055806396972) -- (5.0534340774186255,0.15026659542881624);
\draw [line width=0.5pt] (5.122500053428388,0.18055055806396972) -- (5.112765922581374,0.25533340457118464);
\draw [line width=1.5pt] (1.3484,4.7452)-- (1.7419545938239018,4.46690328536827);
\draw [line width=1.5pt] (1.6129088950882646,4.558156175206353) -- (1.4885735626200263,4.526005208475765);
\draw [line width=1.5pt] (1.6129088950882646,4.558156175206353) -- (1.6017810312038758,4.686098076892505);
\draw [line width=0.5pt] (-0.6,-0.64)-- (-0.02,0.76);
\draw [line width=1.5pt] (0.52,-0.38)-- (-0.6,-0.64);
\draw [line width=1.5pt] (-0.12080631057877562,-0.5287586078129303) -- (-0.06216926377891731,-0.4145016329523566);
\draw [line width=1.5pt] (-0.12080631057877562,-0.5287586078129303) -- (-0.017830736221082746,-0.605498367047644);
\draw [line width=0.5pt] (0.5758,1.1316)-- (1.5104,1.5348);
\draw [line width=1.5pt] (1.5104,1.5348)-- (1.5378,0.778);
\draw [line width=1.5pt] (1.5271014282846413,1.0734992362840672) -- (1.426126370153896,1.1528528574817882);
\draw [line width=1.5pt] (1.5271014282846413,1.0734992362840672) -- (1.6220736298461027,1.1599471425182133);
\draw [line width=1.5pt] (0.8844,2.284)-- (1.3872,1.9758);
\draw [line width=1.5pt] (1.2065255828658927,2.086547524583795) -- (1.0845652563263033,2.046315220249401);
\draw [line width=1.5pt] (1.2065255828658927,2.086547524583795) -- (1.1870347436736979,2.2134847797506003);
\draw [line width=0.5pt] (1.3872,1.9758)-- (0.5726,1.6126);
%\begin{scriptsize}
\draw [fill=black] (0.8844,2.284) circle (2.5pt);
\draw [fill=black] (6.24,3.46) circle (1.0pt);
\draw[color=black] (6.2060504302448205,3.6887649327967647) node {$z_j$};
\draw[color=black] (6.707705264386586,3.2463378427221756) node {$\tilde{\gamma}_j$};
\draw [fill=black] (4.26,1.9) circle (1.0pt);
\draw[color=black] (4.02408046328605,1.9196081838789977) node {$y_j$};
\draw [fill=black] (3.54,4.6) circle (1.0pt);
\draw[color=black] (3.531377567521166,4.865218785949649) node {$x_j$};
\draw [fill=black] (6.523,2.768) circle (1.0pt);
\draw[color=black] (7.0088084871477625,2.673193657852822) node {$x_{j+1}$};
\draw [fill=black] (4.0546,0.7836) circle (1.0pt);
\draw[color=black] (3.6134736907674026,0.7325475582074663) node {$y_{j+1}$};
\draw [fill=black] (6.1116,-0.378) circle (1.0pt);
\draw[color=black] (6.59654688048735,-0.4137408115312413) node {$z_{j+1}$};
\draw [fill=black] (1.7419545938239018,4.46690328536827) circle (2.5pt);
\draw[color=black] (4.968713998882993,4.342350406770589) node {$\tilde{\alpha}_j$};
\draw[color=black] (6.887594942110468,1.2051401316962318) node {$\tilde{\alpha}_{j+1}$};
\end{tikzpicture}

\caption{On the left, the piecewise geodesic $\tilde{\sigma}_n$ is the spiral-shaped path running around the outside of the diagram. 
On the right is a magnified portion of the left-hand picture showing more detail.}
\label{fig:spiral}
\end{figure}
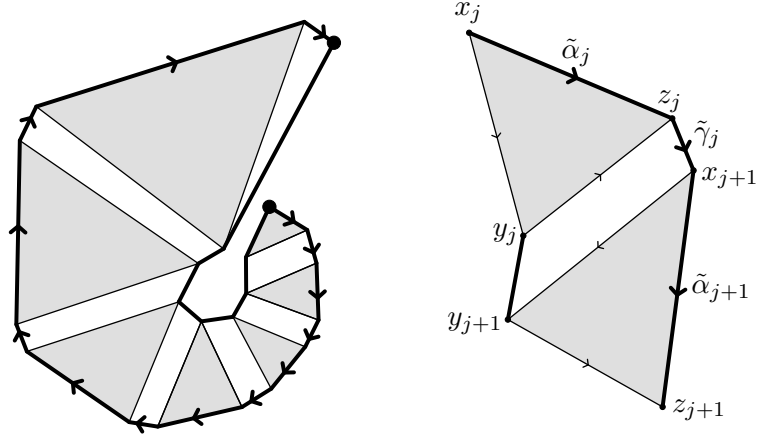

To be more precise, let $M_j$ be the Seifert component of $N$ containing $\gamma_j$, with its given product metric as a hyperbolic surface crossed with a circle of length one.
Let $(\tilde{M}_j, d_j)$ be the Seifert component of $\tilde{N}$ that contains $\tilde{\gamma}_j$, with the product metric $d_j$ induced by lifting the given metric on $M_j$.
Let $\tilde{T}_j$ be the JSJ plane containing $\tilde{\alpha}_j$. Recall that $\tilde{T}_j$ is a topological plane that is a subspace of both $\tilde{M}_{j-1}$ and $\tilde{M}_j$.  However the metrics $d_{j-1}$ and $d_j$ typically do not agree on the common subspace $\tilde{T}_j$.
Each metric space $\overleftarrow{E_j} = (\tilde{T}_j,d_{j-1})$ and $\overrightarrow{E_j} = (\tilde{T}_j, d_{j})$ is a Euclidean plane and the identity map $\overleftarrow{E_j} \to \overrightarrow{E_j}$ is affine.

The plane $\tilde{T}_j$ universally covers a JSJ torus $T_j$ obtained by identifying boundary tori $\overleftarrow{T_j}$ and $\overrightarrow{T_j}$ of Seifert components $M_{j-1}$ and $M_j$. 
The initial and terminal points $x_j$ and $z_j$ of $\tilde{\alpha}_j$ are contained in Euclidean geodesics $\overleftarrow{\ell_j} \subset \overleftarrow{E_j}$ and $\overrightarrow{\ell_j} \subset \overrightarrow{E_j}$ that project to fibers $\overleftarrow{f_j} \subset \overleftarrow{T_j}$ and $\overrightarrow{f_j} \subset \overrightarrow{T_j}$ respectively.
The lines $\overleftarrow{\ell_j}$ and $\overrightarrow{\ell_j}$ intersect in a unique point $y_j$.
Similarly, we consider the subpath $\tilde{\alpha}_{j}^{-1}$ in the double spiral $\tilde{\rho}_n$. 
Let $\bar{y}_j$ be the intersection point of the two fibers that contain its endpoints.

Our goal is to find a linear upper bound for the distance in $(N,d)$ between the endpoints of $\tilde{\rho}_n$.
By the triangle inequality it suffices to produce an upper bound for the distance between successive points of the linear sequence $y_1,\dots,y_{nm},\bar{y}_{nm},\dots,\bar{y}_1$.
Recall that the inclusions $(\tilde{M}_j,d_j) \to (\tilde{N},d)$ are $K$--bilipschitz for some universal constant $K$, as explained in Remark~\ref{rem:MetricOnGraphManifold}.
Thus it is enough to bound the distance between these points with respect to the given product metrics on each Seifert component.

Let
$\eta$ be the maximum of lengths of $\tilde{\gamma}_j$ with respect to metric $d_j$.
Let $A$ be the constant given by Lemma~\ref{lem:ConstructSequence}.
We claim that
\begin{equation}
\label{eqn:HomotopicBounded}
\tag{$\star$}
   d_j(y_{j},y_{j+1}) \le \eta + A.
\end{equation}
We prove this claim by examining the quadrilateral with vertices $y_j$, $z_j$, $x_{j+1}$, and $y_{j+1}$, illustrated on the right-hand side of Figure~\ref{fig:spiral}.  This quadrilateral is a trapezoid in the sense that the opposite sides $[y_j,z_j]$ and $[y_{j+1},x_{j+1}]$ lie in fibers of $\tilde{M}_j$ that are parallel lines in the product metric $d_j$.

To find the lengths of these parallel sides, we must first examine the JSJ triangle $\Delta(x_j,y_j,z_j)$, also shown on the right-hand side of Figure~\ref{fig:spiral}.
We consider this triangle to be a pair of homotopic paths $\tilde{\alpha}_j$ and $\tau_j$ from $x_j$ to $y_j$ in the plane $\tilde{T}_j$, where $\tau_j = [x_j,y_j] \cup [y_j,z_j]$ is a segment of $\overleftarrow{\ell_j}$ concatenated with a segment of $\overrightarrow{\ell_j}$.
In particular, $\tilde{\alpha}_j$ and $\tau_j$ project to homotopic loops in $T_j$ of the form $g(\alpha_j)$ and $(\overleftarrow{f_j})^{r_j} (\overrightarrow{f_j})^{s_j}$ for some $r_j,s_j \in \Z$.
Recall that the fiber $\overleftarrow{f_j}$ has length one in the product metric on $M_{j-1}$.  Similarly the fiber $\overrightarrow{f_j}$ has length one in $M_j$.
It follows that $d_{j-1}(x_j,y_j) = \abs{r_j}$ and $d_j(y_j,z_j) = \abs{s_j}$.
Since $g(\alpha_j) = g(c_j)^{t(j)}$, the homology relation (written additively)
\[
   t(j) \, \bigl[ g(c_j) \bigr] = t(j)\,a_{j}\, [\overleftarrow{f_j}] + t(j)\,b_{j}\,[\overrightarrow{f_j}] \quad\text{in} \quad
   H_1(T_j;\Z) \approx \pi_1(T_j),
\]
implies that $r_j = t(j)\,a_j$ and $s_j = t(j)\,b_j$.

Let $\beta_j$ be the lift of $\gamma_j$ based at $y_j$.
The fiber which contains $y_{j+1}$ and $x_{j+1}$ will intersect $\beta_j$ exactly at one point. We denote this point by $u_j$. 
It follows that $d_{j}(u_j,y_j) \le \eta$ and $d_{j}(u_j,y_{j+1}) = \bigabs{t(j+1)\,a_{j+1} + t(j)\,b_j}$. Using the triangle inequality for $\Delta(y_j,u_j,y_{j+1})$, we have
\begin{align*}
d_{j}(y_j,y_{j+1}) &\le d_{j}(y_j,u_j) + d_{j}(u_j,y_{j+1})\\
&\le \eta + \bigabs{t(j+1)\,a_{j+1} + t(j)\,b_j}\\
&\le \eta + A
\end{align*}
by Lemma~\ref{lem:ConstructSequence}(\ref{item:sequence:ApproxUpper}), completing the proof of \eqref{eqn:HomotopicBounded}.

\usetikzlibrary{arrows}
\begin{figure}
\definecolor{wqwqwq}{rgb}{0.3764705882352941,0.3764705882352941,0.3764705882352941}
\begin{tikzpicture}[line cap=round,line join=round,>=triangle 45,x=1.0cm,y=1.0cm]
\clip(1.826997701735815,-0.9248890880809865) rectangle (13.519613097484152,5.2296509084713145);
\fill[line width=0.pt,color=wqwqwq,fill=wqwqwq,fill opacity=0.2] (6.918181818181819,4.938181818181819) -- (6.818181818181818,2.1781818181818187) -- (4.605454545454546,3.767272727272727) -- cycle;
\fill[color=wqwqwq,fill=wqwqwq,fill opacity=0.2] (6.543636363636364,0.25090909090909086) -- (6.961818181818183,-0.44) -- (7.234545454545454,0.01454545454545451) -- cycle;
\fill[color=wqwqwq,fill=wqwqwq,fill opacity=0.2] (4.058181818181818,0.2254545454545456) -- (5.670909090909096,1.17818181818182) -- (4.58,-0.4763636363636339) -- cycle;
\fill[line width=2.pt,color=wqwqwq,fill=wqwqwq,fill opacity=0.2] (4.363240324032511,3.471233123312465) -- (6.4,1.9745454545454404) -- (3.743636363636367,2.269090909090911) -- cycle;
\fill[line width=0.5pt,color=wqwqwq,fill=wqwqwq,fill opacity=0.2] (3.949090909090918,0.5781818181817883) -- (3.683609345865521,1.8704817534913831) -- (5.709070617936441,1.6698836797772711) -- cycle;
\fill[line width=0.5pt,color=wqwqwq,fill=wqwqwq,fill opacity=0.2] (6.616363636363642,-0.6036363636363611) -- (6.052727272727273,-0.7672727272727284) -- (6.052727272727272,0.2872727272727282) -- cycle;
\fill[line width=0.5pt,color=wqwqwq,fill=wqwqwq,fill opacity=0.2] (5.652727272727274,-0.6945454545454544) -- (4.907548589341696,-0.6651285266457689) -- (5.670909090909096,0.5236363636363657) -- cycle;
\fill[line width=2.pt,color=wqwqwq,fill=wqwqwq,fill opacity=0.2] (8.095865696502162,4.923552204786535) -- (8.127273989655489,2.1619198036293508) -- (10.378792881563966,3.6955542386254048) -- cycle;
\fill[line width=2.pt,color=wqwqwq,fill=wqwqwq,fill opacity=0.2] (10.616719417901304,3.3895495046901027) -- (8.540268256333876,1.9479582712369445) -- (11.20312895590428,2.176426404839135) -- cycle;
\fill[line width=0.pt,color=wqwqwq,fill=wqwqwq,fill opacity=0.2] (10.95573412816692,0.4911427719820856) -- (11.25323564532997,1.7764491286533148) -- (9.223416360942931,1.626227210920047) -- cycle;
\fill[color=wqwqwq,fill=wqwqwq,fill opacity=0.2] (10.83791482423075,0.1412342596651171) -- (9.249351801525389,1.1337291137021142) -- (10.298823880073654,-0.5474049267993321) -- cycle;
\fill[color=wqwqwq,fill=wqwqwq,fill opacity=0.2] (9.221007802226715,-0.7388719207504169) -- (9.966687279135932,-0.7279749699972508) -- (9.233092311738414,0.47938564039624937) -- cycle;
\fill[color=wqwqwq,fill=wqwqwq,fill opacity=0.2] (8.259927062330156,-0.6240533006048894) -- (8.819324625230138,-0.8016403962479197) -- (8.845520469886926,0.252579644078198) -- cycle;
\fill[color=wqwqwq,fill=wqwqwq,fill opacity=0.2] (8.353859559919362,0.22842184623476774) -- (7.91909564302622,-0.4337098272990246) -- (7.657292189006765,0.009293942316121417) -- cycle;
\draw [line width=1.5pt] (6.543636363636364,0.25090909090909086)-- (6.052727272727272,0.2872727272727282);
\draw [line width=1.5pt] (5.670909090909096,0.5236363636363657)-- (6.052727272727272,0.2872727272727282);
\draw [line width=1.5pt] (6.961818181818183,-0.44)-- (6.616363636363642,-0.6036363636363611);
\draw [line width=1.5pt] (6.71374239008285,-0.5575095855588407) -- (6.746910159215588,-0.4327699320813756);
\draw [line width=1.5pt] (6.71374239008285,-0.5575095855588407) -- (6.831271658966236,-0.610866431554984);
\draw [line width=1.5pt] (5.707272727272726,1.6690909090909087)-- (6.4,1.9745454545454542);
\draw [line width=1.5pt] (4.058181818181818,0.2254545454545456)-- (3.9490909090909088,0.5781818181818185);
\draw [line width=1.5pt] (3.9767622716044304,0.488711079388096) -- (4.105011410801259,0.4331712891887704);
\draw [line width=1.5pt] (3.9767622716044304,0.488711079388096) -- (3.902261316471467,0.3704650744475927);
\draw [line width=1.5pt] (5.709070617936442,1.6698836797772718)-- (5.670909090909096,1.17818181818182);
\draw [line width=1.5pt] (6.4,1.9745454545454542)-- (6.818181818181818,2.1781818181818187);
\draw [line width=1.5pt] (4.36,3.4672727272727277)-- (4.6054545454545455,3.7672727272727276);
\draw [line width=1.5pt] (4.535523036744678,3.6818008832940006) -- (4.558987820752411,3.5548777334339725);
\draw [line width=1.5pt] (4.535523036744678,3.6818008832940006) -- (4.406466724702134,3.679667721111482);
\draw [line width=1.5pt] (3.683609345865521,1.8704817534913831)-- (3.743636363636367,2.269090909090911);
\draw [line width=1.5pt] (3.726038301718556,2.15223105394992) -- (3.811057526984035,2.055113530329423);
\draw [line width=1.5pt] (3.726038301718556,2.15223105394992) -- (3.616188182517852,2.0844591322528716);
\draw [line width=1.5pt] (6.052727272727273,-0.7672727272727284)-- (5.652727272727274,-0.6945454545454544);
\draw [line width=1.5pt] (5.770697797549621,-0.7159946408767907) -- (5.870353440947264,-0.6339651656991355);
\draw [line width=1.5pt] (5.770697797549621,-0.7159946408767907) -- (5.835101104507283,-0.8278530161190482);
\draw [line width=1.5pt] (5.670909090909096,1.17818181818182)-- (5.670909090909096,0.5236363636363657);
\draw [line width=1.5pt] (4.907548589341696,-0.6651285266457689)-- (4.58,-0.47636363636363394);
\draw [line width=1.5pt] (4.664970043516121,-0.5253315191862246) -- (4.796757950709068,-0.47880778849085115);
\draw [line width=1.5pt] (4.664970043516121,-0.5253315191862246) -- (4.690790638632626,-0.6626843745185526);
\draw [line width=1.5pt] (6.81045399060447,2.1744187021441537)-- (7.206587745447177,2.427790820484651);
\draw [line width=1.5pt] (7.289090909090914,5.069090909090917)-- (7.725454545454553,5.087272727272727);
\draw [line width=1.5pt] (7.590574758137849,5.081652736134533) -- (7.511374721216847,4.979733963523046);
\draw [line width=1.5pt] (7.590574758137849,5.081652736134533) -- (7.50317073332862,5.176629672840598);
\draw [line width=1.5pt] (4.58,-0.47636363636363394)-- (4.058181818181818,0.2254545454545456);
\draw [line width=1.5pt] (4.2648216513877015,-0.05246523021608637) -- (4.4042451102024405,-0.062140411467468105);
\draw [line width=1.5pt] (4.2648216513877015,-0.05246523021608637) -- (4.233936707979376,-0.18876867944162218);
\draw [line width=0.5pt] (4.058181818181818,0.2254545454545456)-- (5.670909090909097,1.178181818181847);
\draw [line width=0.5pt] (5.670909090909097,1.178181818181847)-- (4.58,-0.47636363636363394);
\draw [line width=1.5pt] (4.605454545454546,3.767272727272727)-- (6.918181818181819,4.938181818181819);
\draw [line width=1.5pt] (5.836202362798583,4.390387188255087) -- (5.806325354714686,4.264818695204977);
\draw [line width=1.5pt] (5.836202362798583,4.390387188255087) -- (5.717311008921677,4.440635850249569);
\draw [line width=0.5pt] (6.918181818181819,4.938181818181819)-- (6.818181818181818,2.1781818181818187);
\draw [line width=1.5pt] (6.918181818181819,4.938181818181819)-- (7.289090909090915,5.069090909090923);
\draw [line width=1.5pt] (7.182257532316632,5.031385011405879) -- (7.136430220091239,4.910720437014234);
\draw [line width=1.5pt] (7.182257532316632,5.031385011405879) -- (7.070842507181495,5.096552290258507);
\draw [line width=1.5pt] (3.743636363636367,2.269090909090911)-- (4.363240324032511,3.471233123312465);
\draw [line width=1.5pt] (4.091635678546343,2.9442716595593126) -- (4.141022467802537,2.8250197115421636);
\draw [line width=1.5pt] (4.091635678546343,2.9442716595593126) -- (3.9658542198663413,2.915304320861212);
\draw [line width=0.5pt] (4.363240324032511,3.471233123312465)-- (6.4,1.9745454545454404);
\draw [line width=0.5pt] (6.4,1.9745454545454404)-- (3.743636363636367,2.269090909090911);
\draw [line width=0.5pt] (3.949090909090918,0.5781818181817883)-- (3.683609345865521,1.8704817534913831);
\draw [line width=0.5pt] (3.683609345865521,1.8704817534913831)-- (5.709070617936441,1.6698836797772711);
\draw [line width=0.5pt] (5.709070617936441,1.6698836797772711)-- (3.949090909090918,0.5781818181817883);
\draw [line width=0.5pt] (5.709070617936441,1.6698836797772711)-- (3.949090909090918,0.5781818181817883);
\draw [line width=1.5pt] (6.616363636363642,-0.6036363636363611)-- (6.052727272727273,-0.7672727272727284);
\draw [line width=1.5pt] (6.254477250242472,-0.7087001531554113) -- (6.30707337271716,-0.5908284858237413);
\draw [line width=1.5pt] (6.254477250242472,-0.7087001531554113) -- (6.362017536373754,-0.7800806050853482);
\draw [line width=0.5pt] (6.052727272727273,-0.7672727272727284)-- (6.052727272727272,0.2872727272727282);
\draw [line width=0.5pt] (6.052727272727272,0.2872727272727282)-- (6.616363636363642,-0.6036363636363611);
\draw [line width=1.5pt] (5.652727272727274,-0.6945454545454544)-- (4.907548589341696,-0.6651285266457689);
\draw [line width=1.5pt] (5.19682850939695,-0.6765482394486506) -- (5.28402463693544,-0.5813804013876145);
\draw [line width=1.5pt] (5.19682850939695,-0.6765482394486506) -- (5.27625122513353,-0.7782935798036107);
\draw [line width=0.5pt] (4.907548589341696,-0.6651285266457689)-- (5.670909090909096,0.5236363636363657);
\draw [line width=0.5pt] (5.670909090909096,0.5236363636363657)-- (5.652727272727274,-0.6945454545454544);
\draw [line width=1.5pt] (7.206587745447176,2.4277908204846548)-- (7.711523220315041,2.418548909723031);
\draw [line width=0.5pt] (7.289090909090931,5.0690909090909235)-- (7.2065877454471705,2.427790820484651);
\draw [line width=0.5pt] (8.095865696502162,4.923552204786535)-- (8.127273989655489,2.1619198036293508);
\draw [line width=0.5pt] (8.127273989655489,2.1619198036293508)-- (10.378792881563966,3.6955542386254048);
\draw [line width=0.5pt] (10.616719417901304,3.3895495046901027)-- (8.540268256333876,1.9479582712369445);
\draw [line width=0.5pt] (8.540268256333876,1.9479582712369445)-- (11.20312895590428,2.176426404839135);
\draw [line width=0.5pt] (11.20312895590428,2.176426404839135)-- (10.616719417901304,3.3895495046901027);
\draw [line width=0.5pt] (10.83791482423075,0.1412342596651171)-- (9.249351801525389,1.1337291137021142);
\draw [line width=0.5pt] (9.249351801525389,1.1337291137021142)-- (10.298823880073654,-0.5474049267993321);
\draw [line width=0.5pt] (9.966687279135932,-0.7279749699972508)-- (9.233092311738414,0.47938564039624937);
\draw [line width=0.5pt] (9.233092311738414,0.47938564039624937)-- (9.221007802226715,-0.7388719207504169);
\draw [line width=1.5pt] (8.259927062330156,-0.6240533006048894)-- (8.819324625230138,-0.8016403962479197);
\draw [line width=1.5pt] (8.61909189892233,-0.7380742481700684) -- (8.509811644083088,-0.8067612772308059);
\draw [line width=1.5pt] (8.61909189892233,-0.7380742481700684) -- (8.569440043477202,-0.6189324196220023);
\draw [line width=0.5pt] (8.819324625230138,-0.8016403962479197)-- (8.845520469886926,0.252579644078198);
\draw [line width=0.5pt] (8.845520469886926,0.252579644078198)-- (8.259927062330156,-0.6240533006048894);
\draw [line width=0.5pt] (8.353859559919362,0.22842184623476774)-- (7.91909564302622,-0.4337098272990246);
\draw [line width=0.5pt] (7.91909564302622,-0.4337098272990246)-- (7.657292189006765,0.009293942316121417);
\draw [line width=1.5pt] (7.657292189006765,0.009293942316121417)-- (8.353859559919362,0.22842184623476774);
\draw [line width=1.5pt] (8.095865696502162,4.923552204786536)-- (7.725454545454553,5.087272727272727);
\draw [line width=1.5pt] (7.834402598557707,5.039118042600155) -- (7.950493984198066,5.095534992526765);
\draw [line width=1.5pt] (7.834402598557707,5.039118042600155) -- (7.870826257758648,4.915289939532498);
\draw [line width=1.5pt] (10.616719417901304,3.3895495046901027)-- (10.378792881563967,3.6955542386254048);
\draw [line width=1.5pt] (10.446579614204714,3.6083716090398386) -- (10.575543112093278,3.603033231827119);
\draw [line width=1.5pt] (10.446579614204714,3.6083716090398386) -- (10.419969187371994,3.4820705114883883);
\draw [line width=1.5pt] (11.25323564532997,1.7764491286533148)-- (11.203128955904281,2.1764264048391344);
\draw [line width=1.5pt] (11.21781868489225,2.0591654591278856) -- (11.325951391613629,1.9886856762392582);
\draw [line width=1.5pt] (11.21781868489225,2.0591654591278856) -- (11.130413209620619,1.9641898572531908);
\draw [line width=0.5pt] (9.223416360942931,1.626227210920047)-- (10.95573412816692,0.4911427719820856);
\draw [line width=1.5pt] (10.83791482423075,0.14123425966511713)-- (10.955734128166931,0.49114277198211553);
\draw [line width=1.5pt] (10.9234300807038,0.3952038145107263) -- (10.990206192829058,0.2847455286813904);
\draw [line width=1.5pt] (10.9234300807038,0.3952038145107263) -- (10.803442759568624,0.34763150296584194);
\draw [line width=1.5pt] (10.298823880073654,-0.5474049267993322)-- (10.83791482423075,0.1412342596651171);
\draw [line width=1.5pt] (10.6197628467841,-0.13743486268373648) -- (10.645956272287092,-0.26382309995026254);
\draw [line width=1.5pt] (10.6197628467841,-0.13743486268373648) -- (10.490782432017312,-0.14234756718395483);
\draw [line width=1.5pt] (9.966687279135934,-0.7279749699972508)-- (10.298823880073655,-0.5474049267993322);
\draw [line width=1.5pt] (10.206004645205603,-0.5978672161176501) -- (10.179818808663732,-0.7242570259265232);
\draw [line width=1.5pt] (10.206004645205603,-0.5978672161176501) -- (10.085692350545857,-0.5511228708700607);
\draw [line width=1.5pt] (9.221007802226715,-0.7388719207504169)-- (9.966687279135932,-0.7279749699972508);
\draw [line width=1.5pt] (9.67721294998966,-0.7322051890212311) -- (9.595287298188946,-0.8319462018291446);
\draw [line width=1.5pt] (9.67721294998966,-0.7322051890212311) -- (9.592407783173702,-0.6349006889185241);
\draw [line width=1.5pt] (8.819324625230138,-0.8016403962479197)-- (9.221007802226714,-0.7388719207504167);
\draw [line width=1.5pt] (9.102540864220229,-0.7573839957422643) -- (9.035378769713855,-0.8676080181713043);
\draw [line width=1.5pt] (9.102540864220229,-0.7573839957422643) -- (9.004953657742993,-0.6729042988270326);
\draw [line width=1.5pt] (7.91909564302622,-0.4337098272990246)-- (8.259927062330156,-0.6240533006048894);
\draw [line width=1.5pt] (8.16230342846283,-0.5695336158146888) -- (8.041468018658598,-0.6149085626065357);
\draw [line width=1.5pt] (8.16230342846283,-0.5695336158146888) -- (8.13755468669778,-0.442854565297378);
\draw [line width=1.5pt] (8.540268256333842,1.9479582712369592)-- (8.127273989655489,2.1619198036293508);
\draw [line width=1.5pt] (9.225194003681967,1.6253900236696759)-- (8.540268256333842,1.9479582712369592);
\draw [line width=1.5pt] (9.22341636094293,1.626227210920048)-- (9.249351801525389,1.1337291137021142);
\draw [line width=1.5pt] (9.249351801525389,1.1337291137021142)-- (9.233092311738416,0.47938564039624953);
\draw [line width=1.5pt] (8.353859559919362,0.22842184623476774)-- (8.845520469886926,0.252579644078198);
\draw [line width=1.5pt] (7.234545454545454,0.01454545454545451)-- (6.961818181818183,-0.44);
\draw [line width=1.5pt] (7.055286089421738,-0.2842201539940742) -- (7.0136902312301475,-0.1620323205562664);
\draw [line width=1.5pt] (7.055286089421738,-0.2842201539940742) -- (7.18267340513349,-0.2634222248982792);
\draw [line width=1.5pt] (6.543636363636364,0.25090909090909086)-- (7.234545454545454,0.01454545454545451);
\draw [line width=0.5pt] (6.543636363636364,0.25090909090909086)-- (6.961818181818183,-0.44);
\draw [line width=1.5pt] (3.949090909090918,0.5781818181817883)-- (3.683609345865521,1.8704817534913831);
\draw [line width=1.5pt] (3.7995726307230777,1.306000575593253) -- (3.912867788099732,1.2441597365472081);
\draw [line width=1.5pt] (3.7995726307230777,1.306000575593253) -- (3.719832466856707,1.2045038351259625);
\draw [line width=0.5pt] (3.683609345865521,1.8704817534913831)-- (5.709070617936441,1.6698836797772711);
\draw [line width=0.5pt] (6.818181818181818,2.1781818181818187)-- (4.605454545454546,3.767272727272727);
\draw [line width=1.5pt] (10.378792881563966,3.6955542386254048)-- (8.095865696502162,4.923552204786535);
\draw [line width=1.5pt] (9.163903567418053,4.349049285092112) -- (9.284006454853047,4.396329074523713);
\draw [line width=1.5pt] (9.163903567418053,4.349049285092112) -- (9.19065212321308,4.222777368888226);
\draw [line width=1.5pt] (8.13490595341996,2.1579658827302968)-- (7.711523220315041,2.418548909723031);
\draw [line width=0.5pt] (7.725454545454553,5.087272727272727)-- (7.711523220315041,2.418548909723031);
\draw [line width=1.5pt] (11.20312895590428,2.176426404839135)-- (10.616719417901304,3.3895495046901027);
\draw [line width=1.5pt] (10.873638954062738,2.858052310456455) -- (10.998636607265867,2.825870502666503);
\draw [line width=1.5pt] (10.873638954062738,2.858052310456455) -- (10.821211766539717,2.7401054068627353);
\draw [line width=0.5pt] (11.25323564532997,1.7764491286533148)-- (9.223416360942931,1.626227210920047);
\draw [line width=1.5pt] (10.95573412816692,0.4911427719820856)-- (11.25323564532997,1.7764491286533148);
\draw [line width=1.5pt] (11.12328593145457,1.2150227706152212) -- (11.20048021982733,1.1115765338468238);
\draw [line width=1.5pt] (11.12328593145457,1.2150227706152212) -- (11.008489553669559,1.1560153667885762);
\draw [line width=1.5pt] (9.233092311738416,0.47938564039624953)-- (8.845520469886926,0.252579644078198);
\draw [line width=1.5pt] (7.657292189006765,0.009293942316121417)-- (7.91909564302622,-0.4337098272990246);
\draw [line width=1.5pt] (7.83061229400152,-0.2839850755270273) -- (7.703366395156269,-0.26233875283739605);
\draw [line width=1.5pt] (7.83061229400152,-0.2839850755270273) -- (7.873021436876715,-0.16207713214550748);
\draw [fill=black] (7.234545454545454,0.01454545454545451) circle (2.5pt);
\draw [fill=black] (7.657292189006765,0.009293942316121417) circle (2.5pt);
\end{tikzpicture}

\caption{The outer path represents the piecewise geodesic $\tilde{\rho}_n$.}
\label{fig:doublespiral}
\end{figure}
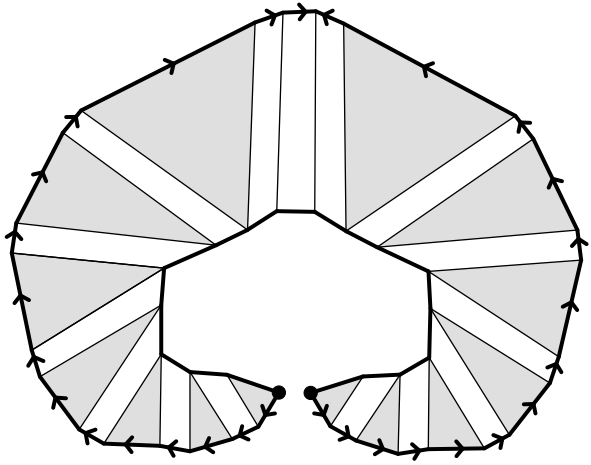

For a similar reason, the distance in the corresponding Seifert component between $\bar{y}_j$ and $\bar{y}_{j-1}$ is at most $\eta + A$. Thus it suffices to find an upper bound for $d_{nm}(y_{nm},\bar{y}_{nm})$.
Let $R$ be the length of $\tilde{c}_1$ in the metric $d_{nm}$.
Let $\bar{z}_{nm}$ be the initial point of $\tilde{\alpha}_{nm}^{-1}$. Since $\alpha_{nm}' = c_{nm+1}^{t(1)}$, the $d_{nm}$--distance between the endpoints of $\tilde{\alpha}_{nm+1}'$ is at most $R\,\bigabs{t(1)}$.
By the triangle inequality, $d_{nm}(z_{nm},\bar{z}_{nm}) \le 2\eta + R\,\bigabs{t(1)}$.

We note that $[z_{nm},y_{nm}]$ and $[\bar{z}_{nm},\bar{y}_{nm}]$ lie in parallel fibers of $\tilde{M}_{nm}$.  Furthermore they are oriented in the same direction with respect to the fiber and have the same length.  Therefore they form opposite sides of a Euclidean parallelogram in $\tilde{M}_{nm}$.  In particular the distances $d_{nm}(y_{nm},\bar{y}_{nm})$ and $d_{nm}(z_{nm},\bar{z}_{nm})$ are equal, so that $d_{nm}(y_{nm},\bar{y}_{nm}) \le 2\eta + R\,\bigabs{t(1)}$.
\end{proof}

\begin{proof}[Proof of Theorem~\ref{thm:lower}]
Any graph manifold has a simple finite cover by Theorem~\ref{thm:KLsimple}.  By Corollary~\ref{cor:ScottFiniteCover} and Proposition~\ref{prop:distortion}, it suffices to prove the theorem for all horizontal surfaces in this cover.  Thus we assume, without loss of generality that $N$ is a simple graph manifold.

Let $\gamma$ be any geodesic loop in $S$ such that $\gamma$ and $\mathcal{T}_g$ have nonempty intersection.  The existence of such a loop is guaranteed by Lemma~\ref{lem:ExistenceOfLoop}.
Replacing $\gamma$ with $\gamma^{-1}$ if necessary, we may assume that $w_\gamma \ge 1$. 
Let $\kappa$ be the maximum of lengths of $\gamma_i$ with respect to the metric $d_S$.

Let $\mu$, $L$ and $C$ be the constants given by Lemma~\ref{lem:LoopsAreQuasigeodesic}.
Let $\bigl \{t(j) \bigr \}$ be the sequence of integers given by Lemma~\ref{lem:ConstructSequence}. For each $n$, let $\sigma_n$ be the spiral loop coresponding to the curve $\gamma$ and the sequence $\bigl \{t(j) \bigr \}$. Let $\rho_n$ be the double spiral of $\sigma_n$. Let $\mathcal{L}$ be the family of lines that are lifts of loops of $\mathcal{T}_g$ or boundary loops of $S$. Since $\tilde{\rho}_n$ satisfies the hypotheses of Lemma~\ref{lem:LoopsAreQuasigeodesic}, it is an $(L,C)$--quasigeodesic in $\tilde S$.

Let $h_n$ be the homotopy class of loop $\rho_n$ at the basepoint $s_0$. We claim that $d_{\tilde{S}}\bigl (\tilde{s}_0,h_{n}(\tilde{s}_0)\bigr ) \succeq n^2 + w_\gamma^n$.
Indeed, let $r$ be the minimum of lengths of $c_i$ with respect to the metric $d_S$. By the construction of the spiral loop $\sigma_n$ we have $\bigabs{\tilde{\sigma}_n} \geq r\sum_{j=1}^{nm} \bigabs{t(j)}$. By Lemma~\ref{lem:ConstructSequence}(\ref{item:sequence:growing}), we have $\bigabs{\tilde{\sigma}_n} \succeq n^2 + w_\gamma^n$. It is obvious that $\abs{\tilde{\rho}_n} \succeq n^2 + w_\gamma^n$ because $\abs{\tilde{\rho}_n}\ge {2\abs{\tilde{\sigma}_n}}$. Since $\tilde{\rho}_n$ is an $(L,C)$--quasigeodesic, it follows that $d_{\tilde{S}}\bigl (\tilde{s}_0,h_{n}(\tilde{s}_0)\bigr ) \succeq n^2 + w_\gamma^n$.

Furthermore, we have $d\bigl (\tilde{x}_0,h_{n}(\tilde{x}_0)\bigr ) \preceq n$ by Lemma~\ref{lem:LinearInSurface}.  
Therefore, $n^2 + w_{\gamma}^{n} \preceq{\Delta_{H}^{G}}$ by Corollary~\ref{cor:GeometricDistortion}. It follows that $\Delta_H^G$ is at least quadratic.
If the horizontal surface is not virtually embedded, then we may choose the geodesic loop $\gamma$ in $S$ such that $w_{\gamma}>1$ by Theorem~\ref{thm:RubinsteinWang}.
In this case $w_\gamma$ is an exponential function, and $w_{\gamma}^{n} \preceq{\Delta_{H}^{G}}$.
\end{proof}

%%%%%%%%%%%%%%%%%%%%%%%%%%%%%%%%%%%%%%%%%%%%%%%%%%%%%%%%%
\section{Upper Bound of Distortion}
\label{sec:UpperBound}
%%%%%%%%%%%%%%%%%%%%%%%%%%%%%%%%%%%%%%%%%%%%%%%%%%%%%%%%%

In this section, we will find the upper bound of the distortion of horizontal surface. The main theorem in this section is the following.
\begin{thm}
\label{thm:UpperBound}
Let $g \colon (S,s_0)\looparrowright (N,x_0)$ be a horizontal surface in a graph manifold $N$. Let $G = \pi_1(N,x_0)$ and $H = g_{*}\bigl(\pi_1(S,s_0) \bigr)$. Then the distortion $\Delta_{H}^{G}$ is at most exponential. Furthermore, if the horizontal surface $g$ is virtually embedded then $\Delta_{H}^{G}$ is at most quadratic.
\end{thm}

\begin{defn}
\label{defn:mapoftrees}
Lift the JSJ decomposition of the simple graph manifold $N$ to the universal cover $\tilde N$, and let $\mathbf T_N$ be the tree dual to this decomposition of $\tilde N$.
Lift the collection $\mathcal T_g$ to the universal cover $\tilde S$.  The tree dual to this decomposition of $\tilde S$ will be denoted by $\mathbf T_S$. The map $\tilde g$ induces a map $\zeta \colon{\mathbf{T}_S} \to \mathbf{T}_N$.
\end{defn}

The following lemma plays an important role in the proof of Theorem~\ref{thm:UpperBound}.

\begin{lem} 
\label{lem:UniqueIntersection}
Let $F$ be a connected compact surface with non-empty boundary and $\chi(F) <0$. Let $M = F \times S^1$.
Let $g \colon (B,b) \looparrowright (M,x)$ be a horizontal surface.
Then each fiber in $\tilde M$ intersects with $\tilde g(\tilde B)$ exactly at one point.
\end{lem}

\begin{proof}
According to Lemma~2.1 in \cite{RW98}, there exists a finite covering map $p \colon{B \times S^1} \to M$ and an embedding $i \colon{B} \to B \times S^1$ given by $i(x) = (x,1)$ such that $g = p \circ i$. Let $\tilde{i} \colon{\tilde{B}} \to \tilde{B} \times \R$ be the lift of $i$ such that $\tilde{i}(\tilde{b}) = (\tilde{b},0)$. Let $\phi \colon \tilde{B} \to B$ and $\Psi \colon \tilde{F} \to F$ be the universal covering maps. Let $\pi \colon \R \to S^1$ be the usual covering space. Since $\tilde{B} \times \R$ and $\tilde{F} \times \R$ both universal cover $M$, there exists a homeomorphism $\omega \colon \bigl (\tilde{B} \times \R,(\tilde{b},0) \bigr ) \to \bigl (\tilde{F} \times \R,\tilde{g}(\tilde{b}) \bigr )$ such that $(\Psi \times \pi) \circ \omega = p \circ (\phi \times \pi)$. By the unique lifting property, we have $\omega \circ \tilde{i} = \tilde{g}$. It follows that $\tilde{g}(\tilde{B}) = \omega \bigl(\tilde{B} \times \{0\} \bigr)$. Since $\omega$ maps each fiber in $\tilde{B} \times \R$ to a fiber in $\tilde{F} \times \R$. It follows that each fiber in $\tilde{F} \times \R$ intersects $\omega \bigl(\tilde{B} \times \{0\}\bigr)$ exactly at one point.
\end{proof}

\begin{prop}
\label{prop:TreeBijective}
The map $\zeta$ is bijective.
\end{prop}

\begin{proof}
A simplicial map between trees is bijective if is locally bijective.
Thus it suffices to show that the map $\zeta$ is locally injective and locally surjective (see \cite{Stallings83} for details).

Suppose by way of contradiction that $\zeta$ is not locally injective.
Then there exist three distinct blocks $\tilde{B}_1$, $\tilde{B}_2$ and $\tilde{B}_3$ in $\tilde{S}$ such that $\tilde{B}_1 \cap \tilde{B}_2 \neq \emptyset$ and $\tilde{B}_2 \cap \tilde{B}_3 \neq \emptyset$ such that the images $\tilde{g}(\tilde{B}_1)$ and $\tilde{g}(\tilde{B}_3)$ lie in the same block $\tilde{M}_1$ of $\tilde{M}$.  Let $\tilde{M}_2$ denote the block containing the image $\tilde{g}(\tilde{B}_2)$.
We denote $\ell_1 = \tilde{B}_1 \cap \tilde{B}_2$ and $\ell_3 = \tilde{B}_2 \cap \tilde{B}_3$.
We have $\tilde{g}(\ell_1)$ and $\tilde{g}(\ell_3)$ are subsets of the JSJ plane $\tilde{T} = \tilde{M}_1 \cap \tilde{M}_2$.
Since the lines $\ell_1$ and $\ell_3$ are disjoint and the map $\tilde{g}$ is an embedding, it follows that $\tilde{g}(\ell_1)$ and $\tilde{g}(\ell_3)$ are disjoint lines in the plane $\tilde{T}$.
By Remark~\ref{rem:nonempty intersection}, any fiber in the plane $\tilde{T}$ intersects $\tilde{g}(\ell_1)$ and $\tilde{g}(\ell_3)$ at distinct points. This contradicts with Lemma~\ref{lem:UniqueIntersection} because $\tilde{g}(\ell_1)$ and $\tilde{g}(\ell_3)$ are subsets of $\tilde{g}(\tilde{B}_2)$. Therefore, $\zeta$ is locally injective.

By Lemma~\ref{lem:UniqueIntersection}, the block $\tilde{g}(\tilde{B})$ must intersect every fiber of the Seifert component $\tilde{M}$ containing it.  In particular, $\tilde{g}(\tilde{B})$ intersects every JSJ plane adjacent to $\tilde{M}$.  Therefore the map $\zeta$ is locally surjective.
\end{proof}

The following corollary is a combination of Lemma~\ref{lem:UniqueIntersection} and Proposition~\ref{prop:TreeBijective}.

\begin{cor}
\label{cor: fiber intersection one point}
Each fiber of $\tilde{N}$ intersects with $\tilde{g}(\tilde{S})$ in one point.
\end{cor}
\begin{rem}
\label{rem:homomorphism:upper}
Let $G$ and $H$ be finitely generated groups with generating sets $\mathcal{A}$ and $\mathcal{B}$ respectively. Let $\phi \colon G \to H$ be a homomorphism. Then there exists a positive number $L$ such that $\bigabs{\phi(g)}_{\mathcal{B}} \le L\abs{g}_{\mathcal{A}}$ for all $g$ in $G$. Indeed, suppose that $\mathcal{A} = \{g_1,g_2,\dots ,g_n \}$ we define $L = \max \bigset{\abs{\phi(g_{i})}_{\mathcal{B}}}{i = 1,2,\dots,n}$. Since $\phi$ is a homomorphism, it is not hard to see that $\bigabs{\phi(g)}_{\mathcal{B}} \le L\abs{g}_{\mathcal{A}}$ for all $g \in G$.
\end{rem}

The following proposition shows that the distortion of a horizontal surface in a trivial Seifert manifold is linear.
\begin{prop}
\label{prop:linear distortion}
Let $F$ be a connected compact surface with non-empty boundary and $\chi(F)<0$.  Let $g \colon (B,b) \looparrowright (M,x)$ be a horizontal surface where $M = F \times S^1$. Let $H = g_{*} \bigl(\pi_1(B,b) \bigr)$ and $G = \pi_1(M,x)$. Then $H \hookrightarrow G$ is a quasi-isometric embedding.
\end{prop}

\begin{proof}
We first choose generating sets for $\pi_1(B)$, $\pi_1(F)$ and $\pi_1(S^1)$. The generating sets of $\pi_1(F)$ and $\pi_1(S^1)$ induce a generating set on $\pi_1(M)$.
Let $g_1 \colon B \to F$ and $g_2 \colon B \to S^1$ be the maps such that $g=(g_1,g_2)$. We have $g_1 \colon B \to F$ is a finite covering map because $g$ is a horizontal surface in $M$. It follows that $g_{1*}\bigl (\pi_1(B)\bigr )$ is a finite index subgroup of $\pi_1(F)$. As a result, $g_{1*}$ is an $(L',0)$--quasi-isometry for some constant $L'$. Since $g_{*} = (g_{1*},g_{2*})$ we have $\bigabs{g_{*}(h)} \ge \bigabs{g_{1*}(h)} \ge \abs{h}\big/L'$ for all $h \in \pi_1(B)$.
Applying Remark~\ref{rem:homomorphism:upper} to the homomorphism $g_{*}$, the constant $L'$ can be enlarged so that we can show that $g_{*}$ is an $(L',0)$--quasi--isometric embedding. \end{proof}

For the rest of this section, we fix $g \colon (S,s_0) \looparrowright (N,x_0)$ a horizontal surface in a simple graph manifold $N$, the metric $d$ which given by Remark~\ref{rem:MetricOnGraphManifold}, and the hyperbolic metric $d_S$ on $S$ which described in Section~\ref{sec:LowerBound}. By Remark~\ref{rem:straightline in torus}, we also assume that for each curve $c$ in $\mathcal{T}_g$ then $g(c)$ is a straight in the JSJ torus $T$ where $T$ is the JSJ torus such that $g(c) \subset T$. We define a metric $d_{\tilde{g}(\tilde{S})}$ on $\tilde{g}(\tilde{S})$ as the following: for any $u = \tilde{g}(x)$ and $v = \tilde{g}(y)$, we define $d_{\tilde{g}(\tilde{S})}(u,v) = d_{\tilde{S}}(x,y)$.
The following corollary follows by combining Proposition~\ref{prop:linear distortion} with several earlier results, using the fact that $S$ has only finitely many blocks, and $N$ has only finitely many Seifert components.

\begin{cor}
\label{cor:uniform constants}
There exist numbers $L$ and $C$ such that the following holds: For each block $\tilde{B}$ in $\tilde{S}$, let $\tilde{M} = \tilde{F} \times \R$ be the Seifert component of $\tilde{N}$ such that $\tilde{g}(\tilde{B}) \subset \tilde{M}$.
The map $\tilde{g}|_{\tilde{B}} \colon \tilde{B} \to \tilde{M} = \tilde{F} \times \R$ can be expressed as a pair of maps
$\tilde{g}_1 \colon  \tilde{B} \to \tilde{F}$ and $\tilde{g}_2 \colon \tilde{B} \to \R$. Then $\tilde{g}_1$ and $\tilde{g}|_{\tilde B}$ are $(L,C)$--quasi-isometric embeddings, and  
\[
    \bigabs{\tilde{g}_{2}(u)-\tilde{g}_{2}(v)} \le L\,d_{M}\bigl (\tilde{g}_{1}(u),\tilde{g}_{1}(v) \bigr ) + C
    \]
    for all $u,v \in \tilde{B}$.
\end{cor}

\begin{proof}
The map $g|_{B} \colon B \to F \times S^1$ can be written as $(g_1,g_2)$. Since the map $g_1 \colon B \to F$ is a finite covering map, the lift $\tilde{g}_1$ is a quasi--isometry. It follows from Proposition~\ref{prop:linear distortion} that $\tilde{g}|_{\tilde{B}}$ is a quasi--isometric embedding. The facts $\tilde{g}_1$ and $\tilde{g}|_{\tilde{B}}$ are quasi--isometrically embedded imply the final claim.
\end{proof}

Now, we describe informally the strategy of the proof of Theorem~\ref{thm:UpperBound} in the case $N$ is a simple graph manifold. For each $n \in \N$, let $u \in \tilde{g}(\tilde{S})$ such that $d(\tilde{x}_0, u) \le n$. We would like to find an upper bound (either quadratic or exponential as appropriate) for $d_{\tilde{S}}(\tilde{x}_0,u)$ in terms of $n$. We first show the existence of a path $\xi$ in $\tilde{N}$ connecting $\tilde{x}_0$ to $u$ and passing through $k$ Seifert components $\tilde{M}_0, \dots ,\tilde{M}_{k-1}$ such that $\xi$ intersects the plane $\tilde{T}_i = \tilde{M}_{i-1} \cap \tilde{M}_{i}$ at exactly one point, denoted by $y_i$. We will show that $k$ is bounded above by a linear function of $n$.
Therefore it suffices to find an upper bound for $d_{\tilde{S}}(\tilde{x}_0,u)$ in terms of $k$.
Let $d_i$ be the given product metric on $\tilde{M}_i$. By Corollary~\ref{cor: fiber intersection one point} the fiber of $\tilde{M}_{i-1}$ passing through $y_i$ intersects $\tilde{g}(\tilde{S})$ in a unique point which is denoted by $x_i$. Similarly, the fiber of $\tilde{M}_{i}$ passing through $y_i$ intersects $\tilde{g}(\tilde{S})$ in one point which is denoted by $z_i$.

On the one hand, proving that the distance in $\tilde{g}(\tilde{S})$ between the endpoints of $\xi$ is dominated by the sum $\sum_{i=1}^{k-1} \bigl (d_{i-1}(y_i,x_i) + d_{i}(y_i,z_i) \bigr )$ is easy. On the other hand, finding a quadratic or exponential upper bound for this sum as a function of $k$ requires more work. Our strategy is to analyze the growth of the sequence of numbers
\begin{multline*}
   d_0(y_1,x_1), d_1(y_1,z_1), d_1(y_2,x_2), \dots,\\
   d_{i-2}(y_{i-1},x_{i-1}),
   d_{i-1}(y_{i-1},z_{i-1}),
   d_{i-1}(y_i,x_i), \dots,\\
   d_{k-2}(y_{k-1},x_{k-1}),d_{k-1}(y_{k-1},z_{k-1})
\end{multline*}
A relation between $d_{i-1}(y_{i-1},z_{i-1})$ and $d_{i-1}(y_i,x_i)$ in the Seifert component $\tilde{M}_{i-1}$ will be described in Lemma~\ref{lem:Inequality1}. The ratio of $d_{i-1}(y_{i-1},z_{i-1})$ to $d_{i-2}(y_{i-1},x_{i-1})$ in the JSJ plane $\tilde{T}_{i-1}$ will be described in Lemma~\ref{lem: application rule of sines}.

\begin{lem}[Crossing a Seifert component]
\label{lem:Inequality1}
There exists a positive constant $L'$ such that the following holds:
For each block $\tilde{B}$ in $\tilde{S}$, let $\tilde{M} = \tilde{F} \times \R$ be the Seifert component such that $\tilde{g}(\tilde{B}) \subset \tilde{M}$. Let $d_M$ be the given product metric on $\tilde{M}$, and let $\tilde{T}$ and $\tilde{T'}$ be two disjoint JSJ planes in the Seifert component $\tilde{M}$. For any two points $y \in \tilde T$ and $y' \in \tilde{T'}$, let $\ell \subset \tilde{T}$ and $\ell' \subset \tilde{T'}$ be the lines that project to the fiber $S^1$ in $M$ such that $y \in \ell$ and $y' \in \ell'$. Let $x$ \textup{(}resp.\ $x'$\textup{)} be the unique intersection point of $\ell$ \textup{(}resp.\ $\ell'$\textup{)} with $\tilde{g}(\tilde{B}) \cap \tilde{M}$ given by Lemma~\ref{lem:UniqueIntersection}. Then
\[
   d_M(y',x') \le d_M(y,x) + L'\, d_M(y,y')
\]
\end{lem}

\begin{proof} 
Let $\rho$ and $K$ be the constants given by Remark~\ref{rem:MetricOnGraphManifold} and Remark~\ref{rem:constants:setting up} respectively, and let $D = \rho/K$. Let $L$ and $C$ be the constants given by Corollary~\ref{cor:uniform constants}. Let $L' = L + 2 + C/D$.

Let $a$ and $b$ be the projection points of $y$ and $y'$ to $\tilde{F}$ respectively. 
We write $\tilde{g}|_{\tilde{B}} = (\tilde{g}_1,\tilde{g}_2)$ where $\tilde{g}_1 \colon \tilde{B} \to \tilde{F}$ and $\tilde{g}_2 \colon \tilde{B} \to \R$. Since $\tilde{g}|_{\tilde{B}}$ is an embedding map and $x,x' \in \tilde{g}(\tilde{B})$, there exist $a',b' \in \tilde{B}$ such that $\tilde{g}(a') = x$ and $\tilde{g}(b') = x'$. It follows that $\tilde{g}_1(a') = a$ and $\tilde{g}_1(b') = b$. By Corollary~\ref{cor:uniform constants} we have
\[
   \bigabs{\tilde{g}_{2}(a') - \tilde{g}_{2}(b')} \le {L\,d_M \bigl (\tilde{g}_{1}(a'),\tilde{g}_{1}(b') \bigr ) + C} = L\,d_{M}(a,b) + C.
\]
Since $\rho \le d(a,b) \le K\,d_M(a,b)$, it follows that $D \le d_M(a,b)$. Therefore
\[
   \bigabs{\tilde{g}_2(a') - \tilde{g}_2(b')}
      \leq (L + C/D)d_{M}(a,b) \le (L + C/D)d_M(y,y').
\]

With respect to the orientation of the factor $\R$ of $\tilde{M}$, let $\Delta(y,a)$ and $\Delta(y',b)$ be the displacements of pairs of points $(y,a)$ and $(y',b)$ respectively. We would like to show that $\bigabs{\Delta(y,a) - \Delta(y',b)} \le 2\,d_M(y,y')$.
Indeed, let $s$ and $t$ be the real numbers such that $y = (a,s)$ and $y' = (b,t)$. We note that $\Delta(y,a) = -s$ if $s \ge 0$ and $\Delta(y,a) = s$ if $s \le 0$ as well as $\Delta(y',b) = -t$ if $t \ge 0$ and $\Delta(y',b) = t$ if $t \le 0$.
Since $d_M(a,b) \le d_M(y,y')$, it follows that $\bigabs{\Delta(y,a) - \Delta(y',b)} \le 2d_M(y,y')$. Moreover, we have that $d_M(y,x) = \bigabs{\tilde{g}_2(a') + \Delta(y,a)}$ and $d_M(y',x') = \bigabs{\tilde{g}_2(b') + \Delta(y',b)}$. Therefore the previous inequalities imply
\begin{align*}
d_M(y',x') - d_M(y,x) &\le \bigabs{\tilde{g}_2(a') - \tilde{g}_2(b')} + \bigabs{\Delta(y,a) - \Delta(y',b)}  \\
&\le (L + 2 + C/D)d_{M}(y,y') = L'\,d_M(y,y'). \qedhere
\end{align*}
\end{proof}

\begin{lem}[Crossing a JSJ plane]
\label{lem: application rule of sines}
Let $\tilde{M}$ and $\tilde{M}'$ be the two adjacent Seifert components. Let $d_M$ and $d_{M'}$ be the given product metrics on $\tilde{M}$ and $\tilde{M}'$ respectively. Let $\tilde{T} = \tilde{M} \cap \tilde{M}'$, and let $\alpha \subset \tilde{g}(\tilde{S}) \cap \tilde{T}$ be the line such that $\alpha$ universally covers a curve $g(c)$ for some $c$ in $\mathcal{T}_g$. We moreover assume that $\alpha$ is a straight line in $(\tilde{T},d_{M})$. For any two points $x$ and $z$ in the line $\alpha$, let $\overleftarrow{\ell} \subset (\tilde{T},d_M )$ and $\overrightarrow{\ell} \subset (\tilde{T},d_{M'} )$ be the Euclidean geodesics such that $x \in \overleftarrow{\ell}$ and $z \in \overrightarrow{\ell}$ and they project to fibers $\overleftarrow{f} \subset \overleftarrow{T}$ and $\overrightarrow{f} \subset \overrightarrow{T}$ respectively. Let $y$ be the unique intersection point of $\overleftarrow{\ell}$ and $\overrightarrow{\ell}$. Let $a$ and $b$ be the integers such that
\[
   \bigl[ g(c) \bigr] = a [\overleftarrow{f}] + b[\overrightarrow{f}] \quad\text{in} \quad
   H_1(T;\Z).
\] where $T$ is the JSJ torus obtained from gluing $\overleftarrow{T}$ to $\overrightarrow{T}$.
Then $d_M(y,x) = \bigabs{a/b}d_{M'}(y,z)$
\end{lem}

\begin{proof}
Let $\kappa$ be the positive constant such that length of the fiber $\overrightarrow{f}$ with respect to the metric $d_M$ equals to $1/\kappa$. Let $\tilde{c}$ be a path lift of $c$ such that $\tilde{g}(\tilde{c}) \subset \alpha$. Let $x'$ and $z'$ be the initial point and the terminal point of $\tilde{g}(\tilde{c})$ respectively. Let $\overleftarrow{\ell'}$ and $\overrightarrow{\ell'}$ be the Euclidean geodesics in $(\tilde{T},d_M )$ and $(\tilde{T},d_{M'} )$ respectively such that $x' \in \overleftarrow{\ell'}$ and $z' \in \overrightarrow{\ell'}$ and both lines $\overleftarrow{\ell'}$ and $\overrightarrow{\ell'}$ project to fibers in $\overleftarrow{T}$ and $\overrightarrow{T}$ respectively. Let $y'$ be the unique intersection point of $\overleftarrow{\ell'}$ and $\overrightarrow{\ell'}$. It was shown in the proof of Lemma~\ref{lem:LinearInSurface} that $d_{M}(y',x') = \abs{a}$ and $d_{M'}(y',z') = \abs{b}$. By the definiton of $\kappa$, it follows that $d_{M'}(y',z') = \kappa\,d_{M}(y',z')$.

In the Euclidean plane $(\tilde{T},d_M)$, consider the similar triangles $\Delta(x,y,z)$ and $\Delta(x',y',z')$. Since $d_{M}(y,x)\big/d_{M}(y,z) = d_{M}(y',x')\big/d_{M}(y',z') = \kappa\,\bigabs{a\big/b}$, it follows that
$d_{M}(y,x) = \bigabs{a\big/b}\,\kappa\,d_{M}(y,z) = \bigabs{a\big/b}\,d_{M'}(y,z)$.
\end{proof}

\begin{proof}[Proof of Theorem~\ref{thm:UpperBound}]
We may assume that $N$ is a simple graph manifold for the same reason as in the first paragraph of the proof of Theorem~\ref{thm:lower}. Let $K$ be the constant given by Remark~\ref{rem:MetricOnGraphManifold}. Let $L$ and $C$ be the constants given by Corollary~\ref{cor:uniform constants}. Moreover, the constant $L$ can be enlarged so that $L \ge K$ and $L$ is greater than the constant $L'$ given by Lemma~\ref{lem:Inequality1}. Moreover, we assume that the base point $s_0$ belongs to a curve in the collection $\mathcal{T}_g$. 

For any $h \in \pi_1(S,s_0)$ such that $d \bigl (\tilde{x}_0,h(\tilde{x}_0) \bigr ) \le n$. We will show that  $d_{\tilde{S}}\bigl(\tilde{s}_0,h(\tilde{s}_0) \bigr)$ is bounded above by either a quadratic or exponential function in term of $n$. The theorem is confirmed by an application of Corollary~\ref{cor:GeometricDistortion}. We consider the following cases:

{\bf Case 1:} 
The points $\tilde{s}_0$ and $h(\tilde{s}_0)$ belong to the same block $\tilde{B}$. In this degenerate case, we will show that $d_{\tilde{S}}\bigl (\tilde{s}_0,h(\tilde{s}_0) \bigr) \preceq n$. Indeed, let $\tilde{M}$ be the Seifert component such that $\tilde{g}(\tilde{B}) \subset \tilde{M}$. Let $d_{\tilde{M}}$ be the given product metric on $\tilde{M}$. Since $\tilde{g}|_{\tilde{B}}$ is an $(L,C)$--quasi--isometric embedding by Corollary~\ref{cor:uniform constants}, it follows that 
\[
   d_{\tilde{g}(\tilde{B})}\bigl (\tilde{x}_0,h(\tilde{x}_0) \bigr ) \le Ld_{\tilde{M}}\bigl (\tilde{x}_0,h(\tilde{x}_0) \bigr ) + C \le L^{2}d\bigl (\tilde{x}_0,h(\tilde{x}_0)\bigr ) + C \le L^{2}n + C.
\]
Therefore,
\[
   d_{\tilde{S}}(\tilde{s}_0,h(\tilde{s}_0)) = d_{\tilde{g}(\tilde{S})}\bigl (\tilde{x}_0,h(\tilde{x}_0) \bigr )
     \leq {d_{\tilde{g}(\tilde{B})}\bigl (\tilde{x}_0,h(\tilde{x}_0) \bigr )}
      \le L^{2}n + C.
\]

{\bf Case 2:} 
The points $\tilde{s}_0$ and $h(\tilde{s}_0)$ belong to distinct blocks of $\tilde{S}$. 
Let $\mathcal{L}$ be the family of lines in $\tilde{S}$ that are lifts of curves of $\mathcal{T}_g$. Since we assume that $s_0$ belongs to a curve in the collection $\mathcal{T}_g$, thus there are distinct lines $\alpha$ and $\alpha'$ in $\mathcal{L}$ such that $\tilde{s}_0 \in \alpha$ and $h(\tilde{s}_0) \in \alpha'$. Let $e$ and $e'$ be the non-oriented edges in the tree $\mathbf{T}_{S}$ corresponding to the lines $\alpha$ and $\alpha'$ respectively. Consider the non backtracking path joining $e$ to $e'$ in the tree $\mathbf{T}_S$, with ordered vertices $v_0,v_1,\dots,v_{k-1}$ where $v_1$ is not a vertex on the edge $e$ and $v_{k-2}$ is not a vertex on the edge $e'$. Consider the corresponding vertices $\zeta(v_{0}),\dots,\zeta(v_{k-1})$ of the tree $\mathbf{T}_{N}$ (see Definition~\ref{defn:mapoftrees}). We denote the Seifert components corresponding to the vertices $\zeta(v_i)$ by $\tilde{M}_i$ with $i = 0,1, \cdots ,k-1$. We note that the Seifert components $\tilde{M}_i$ are distinct because $\zeta$ is injective by Proposition~\ref{prop:TreeBijective}. Let $d_i$ be the given product metric on the Seifert component $\tilde{M}_i$.

For convenience, relabel $\tilde{x}_0$ by $y_0$, and $h(\tilde{x}_0)$ by $y_k$. Let $\gamma$ be a geodesic in $(\tilde{N},d)$ joining $y_0$ to $y_k$.  Replace the path $\gamma$ by a new path $\xi$, described as follows. For $i = 1,2,\dots,k-1$, let $\tilde{T}_i = \tilde{M}_{i-1} \cap \tilde{M}_i$, and let $t_i = \sup \bigset{t \in [0,1]}{\gamma(t) \in \tilde{T}_i}$.  Let $y_i = \gamma(t_i)$.
Let $\xi_i$ be the geodesic segment in $(\tilde{M}_{i},d_i)$ joining $y_i$ to $y_{i+1}$.  Let $\xi$ be the concatenation $\xi_{0} \xi_{1} \cdots \xi_{k-1}$. Since $d_{i-1}(y_{i-1},y_i) \le L\,d(y_{i-1},y_i)$, it follows that $\abs{\xi} \le L\,\abs{\gamma} \le Ln$. Here $|\cdot|$ denotes the length of a path with respect to the metric $d$.

For each $i$, let $\overleftarrow{\ell_i}$ and $\overrightarrow{\ell_i}$ be the Euclidean geodesics in $(\tilde{T}_i,d_{i-1})$ and $(\tilde{T}_i,d_i)$
%\begin{comment}
%This notation makes no sense.  $\tilde{T}_i$ is the same torus mentioned twice with two different metrics.  Shouldn't the two fibered tori be called $\overleftarrow{T}_i$ and $\overrightarrow{T}_i$?  They are glued together to form $T_i$.
%\end{comment}
passing through $y_i$ such that they project to fibers $\overleftarrow{f_i} \subset \overleftarrow{T_i}$ and $\overrightarrow{f_i} \subset \overrightarrow{T_i}$ respectively . Let $\alpha_i = \tilde{g}(\tilde{S})\cap{\tilde{T}_i}$.
By Corollary~\ref{cor: fiber intersection one point}, the lines $\overleftarrow{\ell_i}$ and $\overrightarrow{\ell_i}$ intersect $\alpha_i$ at points, denoted by $x_i$ and $z_i$ respectively.
Let $\rho$ be the minimum distance between two distinct JSJ planes in $\tilde{N}$ (see Remark~\ref{rem:constants:setting up}).

{\bf Claim~1:} There exists a linear function $J$ not depending on the choices of $n$, $h$, or $\xi$ such that
\[
   d_{\tilde{g}(\tilde{S})} \bigl( y_{0},y_{k} \bigr)
   \le J \bigl (\sum_{i=1}^{k-1} d_{i}(y_i,z_i) + n \bigr ).
\]
Let $z_0 = y_0$ and $z_k = y_k$. We first show that for each $i = 0,\dots, k-1$ then
\[
   d_{i}(z_{i},z_{i+1})
   \le L^{2} \bigl( d_{i}(z_{i},y_{i}) +  \abs{\xi_{i}} + d_{i+1}(y_{i+1},z_{i+1}) \bigr)
\]
Indeed, 
\begin{align*}
    d_{i}(z_{i},z_{i+1}) &\le d_{i}(z_{i},y_{i}) + d_{i}(y_{i},y_{i+1}) + d_{i}(y_{i+1},z_{i+1})\\
    &\le d_{i}(z_{i},y_{i}) + L\abs{\xi_{i}} + d_{i}(y_{i+1},z_{i+1})\\
    &\le d_{i}(z_{i},y_{i}) + L\abs{\xi_{i}} + Ld(y_{i+1},z_{i+1})\\
    &\le d_{i}(z_{i},y_{i}) + L\abs{\xi_{i}} + L^{2}d_{i+1}(y_{i+1},z_{i+1}) \qquad \text {because $L \ge 1$}\\
    &\le L^{2} \bigl (d_{i}(z_{i},y_{i}) +  \abs{\xi_{i}} + d_{i+1}(y_{i+1},z_{i+1}) \bigr )
\end{align*}
Using Corollary~\ref{cor:uniform constants} we obtain
\begin{align*}
   d_{\tilde{g}(\tilde{S})}\bigl (y_{0},y_{k}\bigr ) &\le \sum_{i=0}^{k-1}d_{\tilde{g}(\tilde{S})}(z_i,z_{i+1}) \le  \sum_{i=0}^{k-1}Ld_{i}(z_{i},z_{i+1}) +kC\\ &\le 2L^{3}\sum_{i=0}^{k}d_{i}(z_{i},y_{i}) + L^{3}\abs{\xi} +kC \\
   &= 2L^{3}\sum_{i=1}^{k-1}d_{i}(z_{i},y_{i}) + L^{3}\abs{\xi} +kC 
\end{align*}
%The triangle inequality implies
%\[
 %  d_{0}(y_0,z_1) \le d_{0}(y_0,y_1) + d_{0}(y_1,z_1)
%\]
%and 
%\[
 %  d_{k-1}(z_{k-1},y_k) \le d_{k-1}(z_{k-1},y_{k-1}) + d_{k-1}(y_{k-1},y_k).
%\]
%By Corollary~\ref{cor:uniform constants}, we have
%\[
 %  d_{\tilde{g}(\tilde{S})}(\tilde{x}_0,u) \le 2L\,\sum_{i=1}^{k-1}d_{i+1}(y_i,z_i) + L\abs{\beta} + kC.
%\]
%Let $\rho$ be the minimum distance between two distinct JSJ planes in $\tilde{N}$ (see Remark~\ref{rem:constants:setting up}). 
Since $d \bigl (\tilde{x}_0,h(\tilde{x}_0) \bigr ) \le n$, it follows that $k \rho \le n$, and therefore $k \le n/\rho$. We note that $\abs{\xi} \le Ln$. Letting $A = \max \{2L^{3},L^{4} + C/\rho \}$, and $J(x) = Ax$, Claim~1 is confirmed.

In order to complete the proof, it suffices to find an appropriate upper bound for the sum appearing in the conclusion of Claim~1.
In the general case, we need an exponential upper bound, but in the special case of trivial dilation we require a quadratic upper bound.

By Lemma~\ref{lem:Inequality1} we immediately see
\begin{equation}
\tag{$*$}
\label{eq:blacklozenge}
    d_{i-1}\bigl (y_{i},x_{i}\bigr ) \le d_{i-1}\bigl (y_{i-1},z_{i-1}\bigr ) + L^{2}\bigabs{\xi_{i-1}}.
\end{equation}
%We recall that $\mathcal{T}_g$ is the collection of components of $g^{-1}(\mathcal{T})$ in $S$ where $\mathcal{T}$ is the JSJ decomposition of $N$.
Observe that $\alpha_i$ universally covers a closed curve $g(c_i) \in T_i$ for some $c_i \in \mathcal{T}_{g}$ and
$\bigl[ g(c_i) \bigr] = a_{i}\, [\overleftarrow{f_i}] + b_{i}\,[\overrightarrow{f_i}]$ in $H_1(T_i;\Z)$ for some $a_i,b_i \in \Z$. By Lemma~\ref{lem: application rule of sines}, we conclude that
\begin{equation}
\tag{$\dag$}
\label{eq:ddagger}
   d_i(y_i,z_i) = 
   \bigabs{b_i/a_i}\, d_{i-1}(y_i,x_i).
\end{equation}

{\bf Claim 2:}
Suppose $g$ is not virtually embedded.
There exists a function $F$ not depending on the choices of $n$, $h$, or $\xi$ such that
\[
   \sum_{j=1}^{k-1}d_{j}(y_j,z_j) \le F(n)
\]
and $F(n) \sim 2^n$.

Since $g$ is not virtually embedded, the governor $\epsilon = \epsilon(g)$, defined in Remark~\ref{rem:constants:setting up}, must be strictly greater than $1$.
We will show by induction on $j=0,\dots k-1$ that
\[
    d_j(y_j,z_j) \leq L^{3} n \sum_{i=1}^{j}\epsilon^{i}
\]
The base case of $j=0$ is trivial since $y_0=z_0$, so both sides of the inequality equal zero.
For the inductive step, we use \eqref{eq:blacklozenge}, \eqref{eq:ddagger}, and the facts $\abs{\xi_{j-1}} \le Ln$ and $\abs{b_j/a_j} \le \epsilon$ to see that
\begin{align*}
   d_j\bigl( y_j,z_j \bigr) &=\bigabs{b_j/a_j}\, d_{j-1}(y_j,x_j)\\
   &\le \epsilon d_{j-1}(y_j,x_j) \\
   &\le \epsilon \bigl (d_{j-1}(y_{j-1},z_{j-1}) + L^{2}\abs{\xi_{j-1}} \bigr ) \\
   &\le \epsilon d_{j-1}(y_{j-1},z_{j-1}) +\epsilon L^{3}n  \\
   &\le \epsilon(\epsilon + \epsilon^2 + \cdots + \epsilon^{j-1})L^3 n + \epsilon L^3 n \\
   &\le (\epsilon + \epsilon^2 + \cdots + \epsilon^j)L^3 n.
\end{align*}
Summing this geometric series gives
\[
   d_j(y_j,z_j)
   \le \frac{L^{3}n}{\epsilon -1}\epsilon^{j+1}.
\]
Summing a second time over $j$, we obtain
\[
   \sum_{j=1}^{k-1} d_j(y_j,z_j)
      \le \frac{L^{3}n}{\epsilon -1} 
      (\epsilon^2 + \cdots + \epsilon^k)
      \le \frac{\epsilon L^{3}}{(\epsilon-1)^2}
      n \epsilon^k
      \le \frac{\epsilon L^{3}}{(\epsilon-1)^2}n\epsilon^{n/\rho}
\]
which is equivalent to an exponential function of $n$, establishing Claim~2.
(Recall that for any polynomial $p(n)$, we have $p(n)2^n \le 2^n 2^n = 2^{2n}$ for sufficiently large $n$.)

%Note that $n\epsilon^{n/\rho} \le e^{n}\epsilon^{n/\rho} \le e^{ \bigl (1+\ln(\epsilon)/\rho \bigr )n}$. Let $A_{1} = \epsilon L^{3} \bigl /(\epsilon-1)^2$, and $A_{2} = 1 + \ln(\epsilon)/\rho$. We have $\sum_{i=1}^{k-1}d_{i}(y_i,z_i) \le A_{1}e^{A_{2}n}$.

%If $\epsilon = 1$
%\begin{comment}
%There is no need to consider this case, since $\epsilon=1$ is impossible for non--virtually embedded surfaces.
%Please simplify the proof to remove the unneeded case below.
%\end{comment}
%then since $k \le n/\rho$, it is not hard to see $\sum_{i=1}^{k-1}d_{i}(y_i,z_i) \le L^{3}n^{2} \bigl/\rho -L^{3} < 2L^{3}e^{n} \bigl/\rho$. 

%If we choose $B_1 \ge A_{1} + L^{3}/\rho$ and $B_{2} \ge A_{2}$ then Claim~2 is established.

{\bf Claim 3:} Assume that $g$ is virtually embedded. There exists a quadratic function $Q$ not depending on the choices of $n$, $h$, or $\xi$ such that 
\[
\sum_{j=1}^{k-1}d_{j}(y_j,z_j) \le Q(n).
\]
%We recall that $k \le n / \rho$, and thus in order to confirm the above claim, it suffices to show that for each $j = 1,\dots,k-1$
%\[
%d_{j}(y_j,z_j) \le \Lambda L^{3}n
%\]
%(It follows that 
 %  $\sum_{j=1}^{k-1}d_{i}(y_j,z_j) < k\Lambda L^{3}n \le (n/\rho) \Lambda L^{3}n$ which is a quadratic function)
For any $1 \le i \le j$, let 
\[
\Theta_{i,j} = \left| \frac{b_{i}}{a_{i}} \right|\cdot \left| \frac{b_{i+1}}{a_{i+1}} \right|  \cdots  \left| \frac{b_{j}}{a_{j}} \right|
\]
Let $\Lambda$ be the constant given by Proposition~\ref{prop:upperlambda}. In order to prove Claim~3, we mimic the argument of Claim~2---using $\Theta_{i,j}$ in place of the terms of the form $\epsilon^{\ell}$. 
This change gives tighter results than those obtained in Claim~2, since $\Theta_{i,j}$ is bounded above by the constant $\Lambda$.
This upper bound applies only in the virtually embedded case, as explained in Proposition~\ref{prop:upperlambda}.
The key recursive property satisfied by $\Theta_{i,j}$ is the following:
\[
\Theta_{i,j-1} \left|\frac{b_{j}}{a_{j}} \right|  = \Theta_{i,j}.
\]

%Because $d_{0}(y_0,z_0) =0$, and $\Theta_{j,i} \le \Lambda$ for all $1 \le i \le j$ and $\sum_{i=1}^{j}\abs{\beta_i} \le \abs{\beta} \le Ln$, thus to verify the inequality $d_{j}(y_j,z_j) \le \Lambda L^{3}n$ for $j=0,\dots,k-1$, it is enough to show that 
%\begin{equation}
%    \tag{$\ddag$}
%    \label{eq:ddag}
%    d_{j}(y_j,z_j) \le L^{2} \sum_{i=1}^{j}\Theta_{j,i}\abs{\beta_i}
%\end{equation}
We will show by induction on $j =0,\dots,k-1$ that
\[
 d_{j}(y_j,z_j) \le L^{2} \sum_{i=1}^{j}\abs{\xi_{i-1}} \Theta_{i,j}
 \]
As before, the base case $j=0$ is trivial. For the inductive step, we use \eqref{eq:ddagger} and \eqref{eq:blacklozenge} to see that
\begin{align*}
   d_j(y_j,z_j)
   &=  d_{j-1}(y_j,x_j)\,\bigabs{b_j/a_j} \\
   &\le \bigl (d_{j-1}(y_{j-1},z_{j-1}) +  L^{2}\abs{\xi_{j-1}} \bigr )\,\bigabs{b_j/a_j} \\
   &= d_{j-1}(y_{j-1},z_{j-1})\,\bigabs{b_j/a_j} + L^{2}\,\abs{\xi_{j-1}}\,\bigabs{b_j/a_j}\\
   &\le \Bigl( L^{2} \sum_{i=1}^{j-1}\abs{\xi_{i-1}}\,\Theta_{i,j-1} \Bigr ) \,\bigabs{b_j/a_j}  + L^{2}\,\abs{\xi_{j-1}}\,\bigabs{b_j/a_j} \\
   &= L^{2}\sum_{i=1}^{j-1}\abs{\xi_{i-1}}\Theta_{i,j} + L^{2}\abs{\xi_{j-1}}\Theta_{j,j} \\
   &= L^{2}\sum_{i=1}^{j}\abs{\xi_{i-1}}\Theta_{i,j}
\end{align*}
Since $\Theta_{i,j}$ is bounded above by $\Lambda$, and $\sum_{i=1}^{j}\abs{\xi_{i-1}} \le \abs{\xi} \le Ln$, we have 
\[
d_{j}(y_j,z_j) \le L^{2}\sum_{i=1}^{j}\abs{\xi_{i-1}}\Theta_{i,j} \le \Lambda L^{3}n.
\]
Summing over $j$, we obtain
\[
   \sum_{j=1}^{k-1}d_{j}(y_j,z_j) \le (k-1)\Lambda L^{3} n \le \left(\frac{n}{\rho} -1 \right) \Lambda L^{3}n
\]
which is a quadratic function of $n$, establishing Claim~3.

If $g$ is not virtually embedded, we combine Claim~1 and Claim~2 to get an exponential upper bound for $d_{\tilde{g}(\tilde{S})}(y_0,y_k)$. In the virtually embedded case, we combine Claim~1 and Claim~3 to get a quadratic upper bound.
The theorem now follows from Corollary~\ref{cor:GeometricDistortion}.
\end{proof}

%%%%%%%%%%%%%%%%%%%%%%%%%%%%%%%%%%%%%%%%%%%%%%%%%%%%%%%
\section{Fiber surfaces have quadratic distortion}
\label{sec:FiberQuadratic}

In this section we show in detail how results of Gersten and Kapovich--Leeb \cite{Gersten94Quadratic,KapovichLeeb98} can be combined with Thurston's geometric description of $3$--manifolds that fiber over the circle to prove Theorem~\ref{thm:FiberQuadratic}.
This section is an elaboration of ideas that are implicitly used by Kapovich--Leeb in \cite{KapovichLeeb98} but not stated explicitly there.
As explained in the introduction, Theorem~\ref{thm:FiberQuadratic} is the main step in an alternate proof of the virtually embedded case of Theorem~\ref{thm:main thm}.

The following theorem relates distortion of normal subgroups to the notion of divergence of groups.
Roughly speaking, divergence is a geometric invariant that measures the circumference of a ball of radius $n$ as a function of $n$.
(See \cite{Gersten94Quadratic} for a precise definition.)

\begin{thm}[\cite{Gersten94Quadratic}, Thm.~4.1]
\label{thm:GerstenLower}
If $G = H \rtimes_\phi \Z$, where $G$ and $H$ are finitely generated, then the divergence of $G$ is dominated by the distortion $\Delta^G_H$.
\end{thm}

Let $H$ be generated by a finite set $\mathcal{T}$.  An automorphism $\phi \in \Aut(H)$ has \emph{polynomial growth} of degree at most $d$ if there exist constants $\alpha,\beta$ such that \[
   \bigabs{\phi^i(t)}_{\mathcal{T}} \le \alpha n^d + \beta
\]
for all $t \in \mathcal{T}$ and all $i$ with $\abs{i} \le n$.

Gersten claims the following result in the case that $H$ is a free group.  However his proof uses only that $H$ is finitely generated, so we get the following result using the same proof.

\begin{thm}[\cite{Gersten94Quadratic}, Prop.~4.2]
\label{thm:GerstenUpper}
If $G = H \rtimes_\phi \Z$, where $G$ and $H$ are finitely generated and $\phi \in \Aut(H)$ has polynomial growth of degree at most $d$, then $\Delta_H^G \preceq n^{d+1}$.
\end{thm}

\begin{proof}[Proof of Theorem~\ref{thm:FiberQuadratic}]
Let $N$ be a graph manifold that fibers over $S^1$ with surface fiber $S$.  Then $N$ is the mapping torus of a homeomorphism $f \in \Aut(S)$.
In particular, if we let $G = \pi_1(N)$ and $H = \pi_1(S)$, then $G = H \rtimes_\phi \Z$, where $\phi \in \Aut(H)$ is an automorphism induced by $f$.  Passing to finite covers, we may assume without loss of generality that $N$ and $S$ are orientable and that the map $f$ is orientation preserving.

Kapovich--Leeb show that the divergence of the fundamental group of any graph manifold is at least quadratic \cite[\S 3]{KapovichLeeb98}. Therefore by Theorem~\ref{thm:GerstenLower} the distortion of $H$ in $G$ is also at least quadratic.

By Theorem~\ref{thm:GerstenUpper}, in order to establish a quadratic upper bound for $\Delta_H^G$, it suffices to show that $\phi$ has linear growth.
We will first apply the Nielsen--Thurston classification of surface homeomorphisms to the map $f$ (see, for example, Corollary~13.2 of \cite{FarbMargalit12}).
By the Nielsen--Thurston theorem, there exists a multicurve $\{c_1,\dots,c_k\}$ with the following properties. The curves $c_i$ have disjoint closed annular neighborhoods $S_1,\dots,S_k$.  Let $S_{k+1}, \dots, S_{k+\ell}$ be the closures of the connected components of $S - \bigcup_{i=1}^k S_i$.  Then there is a map $g$ isotopic to $f$ and a positive number $m$ such that $g^m$ leaves each subsurface $S_i$ invariant. Furthermore $g^m$ is a product of homeomorphisms $g_1\cdots g_{k+\ell}$ such that each $g_i$ is supported on $S_i$.  For $i=1,\dots,k$, the map $g_i$ (supported on the annulus $S_i$) is a power of a Dehn twist about $c_i$.  For $i = k+1,\dots,k+\ell$, each $g_i$ is either the identity or a map that restricts to a pseudo-Anosov map of $S_i$.

Consider the mapping torus $\hat{N}$ for the homeomorphism $g^m$ of $S$, which finitely covers $N$.  Apply Thurston's geometric classification of mapping tori to $\hat{N}$ to conclude that the family of tori $c_i \times S^1$ in the mapping torus $\hat{N}$ is equal to the family of JSJ tori of $\hat{N}$.  It follows that for each $i = k+1,\dots,k+\ell$ the map $g_i$ is equal to the identity on $S$. Indeed if any $g_i$ were pseudo-Anosov, then the corresponding JSJ component of $\hat{N}$ would be atoroidal and hyperbolic, which is impossible in a graph manifold.

Therefore $g^m$ is a product of powers of Dehn twists about disjoint curves.  In particular the automorphism $\phi^m$ has linear growth.  Clearly $\phi$ itself must also have linear growth.  Therefore the distortion of $H$ in $G$ is at most quadratic as desired.
\end{proof}

%%%%%%%%%%%%%%%%%%%%%%%%%%%%%%%%%%%%%%%%%%%%%%%%%%%%%%%

\bibliographystyle{alpha}
\bibliography{Hruska_Nguyen}
\end{document}